\documentclass[10pt]{article}
\usepackage{graphicx} 

\usepackage[margin=3.5cm]{geometry}

\usepackage{amsfonts}
\usepackage{amsmath}
\usepackage{amssymb}
\usepackage{amsthm}
\usepackage{algpseudocode}
\usepackage{algorithm}
\usepackage{enumitem}
\usepackage[normalem]{ulem} 
\usepackage{xcolor}
\usepackage{hyperref}
\usepackage{mathtools}
\usepackage{tikz}
\usepackage[T1]{fontenc}

\newcommand{\tr}{\mathrm{tr}}
\newcommand{\Ess}{\mathbf{S}}

\newcommand{\ess}{\mathbf{s}}

\theoremstyle{plain}
\newtheorem{theorem}{Theorem}[section]
\newtheorem{proposition}[theorem]{Proposition}
\newtheorem{lemma}[theorem]{Lemma}
\newtheorem{corollary}[theorem]{Corollary}

\theoremstyle{definition}
\newtheorem{remark}[theorem]{Remark}
\newtheorem{definition}[theorem]{Definition}
\newtheorem{example}[theorem]{Example}

\title{Sketching Outside the Hilbert Space for the Infinite-Dimensional Randomised SVD}
\author{Morgan Lynch-Davies\footnote{Department of Mathematics, The University of Manchester, M13 9PL Manchester, UK, \texttt{morgan.lynch-davies@manchester.ac.uk}} \and Marcus Webb\footnote{Department of Mathematics, The University of Manchester, M13 9PL Manchester, UK}}
\date{October 2026}

\begin{document}

\maketitle

\begin{abstract}
In randomised numerical linear algebra, common choices for distributions of sketch vectors $\omega \in \mathbb R^n$ are isotropic, so that $\mathbb E[\omega \omega^T]=I_n$. Unfortunately, the natural extension of an isotropic vector in $(\mathbb R^n, \|\cdot \|_{2})$ to an infinite-dimensional separable Hilbert space $\mathcal H$ may not exist as an element of the space, posing an obstacle for any attempt to lift the randomised SVD to infinite-dimensional target operators $A$. Current workarounds either forego isotropy and take $\omega$ to be a continuous Gaussian process with trace-class covariance operators, or leave $\omega$ implicit by instead sampling $A\omega$ as if $\omega$ were isotropic Gaussian. We unify their approaches by introducing ``sketch functionals'', an example of cylindrical random variables (Vakhania--Tarieladze--Chobanyan~\cite{vakhania1987probability}) which allow us to treat general, non-Gaussian distributions with bounded covariance operators. To demonstrate the utility of the new theory, we apply it to develop novel bounds for the infinite-dimensional randomised SVD with a general class of sub-Gaussian sketch distributions.
\end{abstract}

\noindent\textbf{2020 Mathematics Subject Classification.} Primary 65F55; Secondary 60B12, 60G20, 68W20, 46N40.

\medskip
\noindent\textbf{Keywords.} Randomised numerical linear algebra, sub-Gaussian processes, cylindrical random variables.

\section{Introduction}

The main contribution of this paper is the development of a framework for infinite-dimensional randomised SVD and Nystr\"om approximation for \emph{general} sketching distributions. This allows us to treat distributions beyond those considered in the literature so far, which has been restricted to the isotropic Gaussian distribution \cite{KRESSNER2026259} and Gaussian processes with trace class covariance operators \cite{boulle2021generalization}\cite{boulle2023learning}. To demonstrate the applicability of this new framework, we apply it to sub-Gaussian random processes, improving upon and extending the theory of Saibaba--Mi{\k{e}}dlar~\cite{saibaba2025randomized} to infinite dimensions.

First, a brief history of the randomised SVD for the uninitiated. In short, it is a lightning fast --- hence popular and well-studied --- randomised algorithm for low rank approximation of large matrices. The basic form of the randomised SVD is given in Algorithm~\ref{alg:hmt}, popularised by Halko, Martinsson, Tropp 2011~\cite{halko2011finding}. The first step is commonly referred to as sketching, and is the main focus of this paper. For a target operator $A$ in $\mathbb R^{m \times n}$, we ``sketch'' $\ell$ random vectors $\{\omega^{(i)}\}_{i \le \ell}$ in $\mathbb R^n$ with $\ell \ll m,n$ and construct an orthonormal basis $Q$ for the subspace spanned by their images $\{A \omega^{(i)}\}_{i=1}^{\ell}$. We then compute the SVD of the small matrix $Q^TA$, and $(Q \tilde U) \tilde \Sigma \tilde V^T$ is an approximate SVD for $A$ of rank at most $\ell$. The dominant computational expenditures are likely in step 2 and step 4, namely applying $A$ a total of $2\ell$ times. This is often a remarkably small cost for the quality of approximation, quantified probabilistically in Theorem~\ref{theorem:hmt_maintheorem}.

\begin{algorithm}

\caption{Randomised SVD: finite dimensions}
\label{alg:hmt}
\textbf{Input:} Target matrix $A \in \mathbb R^{m \times n}$, number of samples $\ell$ 

\textbf{Output:} $\hat A$, rank-$\ell$ approximation to $A$
    \begin{algorithmic}[1]
        \State $\Omega \gets (\omega^{(1)}, \cdots, \omega^{(\ell)})$ where $\omega^{(i)} \overset{iid} \sim \mathcal N(0, I_n)$ are $n$-dimensional random vectors
        \State $Y \gets A \Omega$ \label{alg:hmt:step2}
        \State $Q \gets \text{orth}(Y)$
        \State $\tilde U, \tilde \Sigma, \tilde V \gets \mathrm{svd}(Q^TA)$
        \State $\hat A \gets Q Q^T A=(Q\tilde U) \tilde \Sigma \tilde V^T$\label{alg:hmt:step4}
    \end{algorithmic}
\end{algorithm}

\begin{theorem}[Halko, Martinsson, Tropp~\cite{halko2011finding}]
\label{theorem:hmt_maintheorem}
    Let $A \in \mathbb R^{m \times n}$ be the target operator with singular values $\sigma_1, \dots, \sigma_{\min(m,n)}$, $k \in \mathbb N$ the target rank, and $p \ge 2$ a small oversampling parameter. Let $\hat A$ be the output of Algorithm~
    \ref{alg:hmt} applied with $A$ and $\ell=k+p$.
    Define $\Sigma_\perp=\mathrm{diag}(\sigma_{k+1},\dots,\sigma_{\min(m,n)})$. Then
    \begin{equation}
        \mathbb E[\| A - \hat A \|_{F}] \le \sqrt{ 1+\frac{k}{p-1}}\| \Sigma_\perp\|_{F}.
    \end{equation}
    Further, if $p \ge 4$ then with probability at least $1-5t^{-p}-2 e^{-u^2/2}$,
    \begin{equation}
        \| A - \hat A\|_F \le \left(1+t\sqrt{\frac{12k}{p}}\right) \|\Sigma_\perp \|_F + ut\frac{e \sqrt{k+p}}{p+1}\| \Sigma_\perp\|_2.
    \end{equation}
\end{theorem}

To parse this theorem, keep in mind that $\| \Sigma_\perp\|_F$ is the optimal Frobenius error for a rank-$k$ approximation to $A$. So the statement is that with high probability, the random rank-$\ell$ approximation $\hat A$ is worse than optimal by a small factor. Most error bounds for the randomised SVD take a similar form.

Since publication, there has been steady interest in this algorithm with the original paper being one of the most cited in NLA. One good and well-researched question is what other distributions on $\omega$ work well. The Rademacher distribution takes values $\pm 1$ with equal probability and is a common alternative to Gaussian components due to the simplicity and speed of forming $\Omega$. It is also possible to optimise forming $Y=A\Omega$; modifying distributions by a sparsity parameter and letting the components take the value $\pm q^{-1/2}$ with probability $q<1$ and 0 otherwise is one approach that reduces the computational cost of this step to a fraction $q$ while retaining isotropy. Distributions such as these, and many other common choices, fall into the class of sub-Gaussian random variables which were investigated in the context of the randomised SVD in~\cite{saibaba2025randomized}. In Sections 3 and 4, we will take that sub-Gaussian work further, and generalise to infinite dimensions. 

The sketch distribution is usually chosen to be \textbf{isotropic} (defined below). Allowing more general covariance matrices $C$ can be used to introduce a bias into the sketch vector, which Boull\'{e}--Townsend\cite{boulle2021generalization}\cite{boulle2023learning} exploit to improve performance if there is some prior knowledge of the dominant right singular vectors.

\begin{definition}The covariance matrix $C$ of a centred random vector $\omega = (\omega_1, \dots \omega_n)^T  \in \mathbb R^n$ is defined as $\mathbb E[ \omega \omega^T]$. 
The random vector is said to be isotropic if $C=I_n$, the identity.
\end{definition}

This paper is concerned with infinite-dimensional separable Hilbert spaces. This comes with a problem for isotropy: 
\begin{center}
    \textit{Infinite-dimensional sketch vectors $\omega$ are not necessarily elements of $\mathcal H$}.
\end{center}

We refer to this problem as the White Noise Problem. For the purposes of this paper, we define white noise to be an infinite-dimensional, centred, isotropic random variable.

\begin{proposition}[The White Noise Problem~\cite{hida2013white}] \label{prop:whitenoiseproblem} Let $\mathcal H$ be a real infinite-dimensional separable Hilbert space, and let $\omega$ be a random vector in $\mathcal H$ so that $\mathbb E[\omega \omega^T]= I_{\mathcal H}$. Then $\mathbb E[\| \omega\|^2]=\infty$.
    \begin{proof}
    We can explicitly calculate the expectation.
        \begin{align*}
            \mathbb E[\|\omega\|^2] &= \mathbb E[\omega^T \omega] \\
            &= \mathbb E[\tr( \omega^T \omega)]\\
            &= \mathbb E[ \tr( \omega \omega^T)]\\
            &= \tr\big( \mathbb E[ \omega \omega^T]\big) = \tr(I_{\mathcal H}) = \mathrm{dim}(\mathcal H) = \infty\text{,}  
        \end{align*}
    with the infinite-dimensional trace as given in Definition~\ref{def:linear_map_stuff}.
    \end{proof}
\end{proposition}

There are examples of isotropic distributions where, while $\mathbb{E}[\|\omega\|^2] = \infty$, one has $\| \omega\| < \infty$ almost surely (see Example~\ref{example:asfinite}), but if one samples i.i.d. components of an orthogonal expansion then $\|\omega\| = \infty$ almost surely.

Notice that the proof of Proposition~\ref{prop:whitenoiseproblem} implies that $\mathbb E[\|\omega\|^2]$ coincides with the trace of the covariance operator, and this is what Boull\'{e}--Townsend~\cite{boulle2021generalization}\cite{boulle2023learning} leveraged for the first infinite-dimensional randomised SVD. They derived bounds for when sketch vectors $\omega$ are taken to be Gaussian processes with \emph{trace-class} covariance operators $C$, although their bounds are sub-optimal compared with well-established theory for the isotropic case. Nonetheless, allowing general covariance operators has genuine benefit; if it is possible to choose $C$ so it ``aligns well'' with your target operator, the approximation is improved.

More recently Kressner--Persson--Uschmajew~\cite{KRESSNER2026259} made use of the fact that if $\omega$ is a Gaussian process with covariance operator $C$ and $A$ is a linear map, then $A \omega$ is a Gaussian process with covariance operator $A C A^T$. If $C=I$ and $A$ is a Hilbert--Schmidt operator, $ACA^T$ is trace class, so $A\omega$ lies in the Hilbert space with finite expected norm. Skipping directly to step 2 in Algorithm~\ref{alg:hmt} by sampling columns of $A \Omega$ instead of $\Omega$, they matched known finite-dimensional bounds almost exactly. 

In the next section, we detail a class of ``sketch functionals'' which encompasses both of these approaches, treating directly the $\omega$ left implicit by~\cite{KRESSNER2026259}. In the isotropic Gaussian case their bounds are recovered, and we also extend the results of~\cite{boulle2021generalization} to bounded covariance operators. In Section 3, we further exhibit the utility of sketch functionals to give the first results for the infinite-dimensional sub-Gaussian randomised SVD. We compare these bounds with numerical experiments in Section 4 making use of Chebfun~\cite{driscoll2014chebfun}, and conclude in Section 5. We will often use the notation of quasimatrices (Townsend--Trefethen~\cite{townsend2015continuous}), for example with $x,y \in \mathcal H$ we may write $x^Ty=\langle x,y\rangle_{\mathcal H}$ and $x y^T : \mathcal{H} \to \mathcal{H}$ is a linear operator mapping $w \in \mathcal{H}$ to $\langle y, w\rangle_{\mathcal{H}}x$. Over the course of this paper we work in a probability space $(\mathcal G, \mathcal F, \mathbb P)$.
 
\section{Infinite-Dimensional Randomised NLA with Sketch Functionals}

The central object of the paper is the sketch functional, a notion under which the approaches of Kressner et al.~\cite{KRESSNER2026259} and Boull\'{e}--Townsend~\cite{boulle2021generalization}\cite{boulle2023learning} can be viewed. Sketch functionals are a special case of cylindrical random variables, a classical object in probability on Banach spaces~\cite{vakhania1987probability}. This section will define sketch functionals and collect useful properties, many of which are classical. We will then outline the infinite-dimensional randomised SVD with Gaussian sketch in this language for intuition.

We recall the notion of a random embedding for motivation. The phrase is often used in randomised NLA, and indeed Balabanov--Nouy--Pasco~\cite{balabanov2021preconditioners} have already applied it in the infinite-dimensional case. A random embedding randomly maps an element of a high-dimensional space to an element of a lower-dimensional space while approximately preserving Euclidean geometry. That is, a random embedding $\Psi:\mathcal H \to \mathbb R^{\ell}$ should be such that $\langle x,y \rangle_{\mathcal H} \approx \langle \Psi x, \Psi y\rangle_{\mathbb R^\ell}$.

For a high-dimensional space $\mathcal H=\mathbb R^N$, suppose $\omega$ is an isotropic random vector. Then
$$
\mathbb E[(x^T\omega)(\omega^Ty)]=x^T\mathbb E[\omega \omega^T]y= x^Ty\text{.}
$$

Suppose now $\Omega = (\omega^{(1)}, \omega^{(2)},\cdots,\omega^{(\ell)})$ is a matrix with i.i.d. isotropic columns. If $\Psi = \Omega^T/\sqrt{\ell}$, then
\begin{align*}
\langle \Psi x, \Psi y\rangle_{\mathbb R^\ell} &= \frac{1}{\ell} x^T \Omega \Omega^Ty \\
&= \frac{1}{\ell} \sum_{i=1}^{\ell} x^T \omega^{(i)}{\omega^{(i)}}^Ty,
\end{align*}
which is now subject to the strong law of large numbers, and converges to $\langle x,y\rangle_{\mathcal H}$ almost surely as $\ell \to \infty$. Thus, this $\Psi$ is a sensible choice of random embedding.

The crucial part of this argument was that $\mathbb E[(x^T\omega)(\omega^Ty)]= x^Ty$, equivalently written as $\mathbb E[(x^T\omega)^2]= \|x\|_\mathcal H^2$. Notice that we do not have to deal with $\omega$ directly as an element of $\mathcal H$ here, we only need to ``test'' it against $x \in \mathcal H$ by taking an inner product. It is exactly the uniform behaviour of this inner product over $x\in\mathcal H$ from which the quality of $\Psi$ as a random embedding follows.

Our idea then is rather than to require sketch vectors $\omega \in \mathcal H$, we consider inner products of $\omega$ with test vectors $x \in \mathcal H$ and ask that $\langle \omega, x \rangle$ has finite second moments. This is exactly the motivation behind cylindrical random variables~\cite{vakhania1987probability}.

\subsection{Definitions and Concepts from Functional Analysis}

We will use the following concepts from functional analysis, which can be found in~\cite{brezis2011functional}. Let $\mathcal H, \mathcal K$ be real separable Hilbert spaces.

\begin{definition}
A real vector space $\mathcal H$ is a Hilbert space if it is complete and has an inner product $\langle \cdot,\cdot\rangle_{\mathcal H}$ which is symmetric, bilinear and positive definite. Its norm is given by $\|x\|^2_{\mathcal H}=\langle x, x\rangle_{\mathcal H}$. The unit ball $B_{\mathcal H}$ is defined as the set of all elements of $\mathcal H$ with norm at most $1$.
\end{definition}

\begin{definition}\label{def:linear_map_stuff}
Let $A:\mathcal H \to \mathcal K$ be a linear operator. 
\begin{enumerate}
    \item We say $A$ is bounded and write $A \in \mathcal B(\mathcal H, \mathcal K)$ if
\begin{equation}
    \|A\|_{\mathcal B(\mathcal H, \mathcal K)} = \sup_{h \in B_{\mathcal H}} \|Ah\|_{\mathcal K} < \infty.
\end{equation}
\item Suppose $A \in \mathcal B(\mathcal H, \mathcal H)$. We say $A$ is symmetric positive semi-definite or SPSD if $\langle x, Ay\rangle_{\mathcal H} = \langle Ax, y \rangle_{\mathcal H}$ and $\langle x, Ax\rangle_{\mathcal H} \ge 0$ for all $x,y\in \mathcal H$. 
\item Let $A$ be SPSD and let $\{ p_i\}$ be any orthonormal basis for $\mathcal H$. Define the trace of $A$ by $\tr(A)= \sum \langle Ap_i, p_i\rangle$. Say $A$ is trace-class and write $A \in \mathrm{TC}(\mathcal H)$ if $\tr(A)<\infty$.
\item Let $A\in \mathcal B(\mathcal H, \mathcal K)$. Say $A$ is Hilbert--Schmidt and write $A\in \mathrm{HS}(\mathcal H, \mathcal K)$ if $A^TA \in \mathrm{TC}(\mathcal H)$. In this case, define $\|A\|_{HS}^2 = \tr(A^TA)$. Note this coincides with the Frobenius norm in finite dimensions.
\item Suppose $A \in \mathcal B(\mathcal H, \mathcal K)$. Say $A$ is compact if the closure of $A B_\mathcal H$ is compact in $\mathcal K$.
\end{enumerate}
\end{definition}

Items 3, 4 and 5 are all notions of decay on $A$, or specifically decay on sequences $Ap_i$ where $\{p_i\}_{i \in \mathbb N}$ is an orthonormal set. We remark that they are written in decreasing order of strength; trace-class operators are Hilbert--Schmidt operators which are compact operators.

One core result of functional analysis is the spectral theorem for compact self-adjoint operators.

\begin{theorem}[{\cite{fabian2001functional}}]
    Let $A$ be a compact self-adjoint operator on $\mathcal H$. Then $A$ has an eigendecomposition
    $A = \sum \lambda_i u_i u_i^T$ with $u_i$ orthonormal and $\lambda_i \in \mathbb R$, and the only possible limit point of $\{\lambda_i\}$ is $0$.
\end{theorem}
While compactness is a sufficient condition for diagonalisability, we remark it isn't necessary (for example, $I_{\mathcal H}$ is not compact but is already diagonalised), and the theory discussed in this paper applies also to non-compact self-adjoint operators that have eigendecompositions.

When $A$ has an eigendecomposition, the formula for the trace given before is simply the sum of the eigenvalues, and so trace-class operators are precisely operators whose eigenvalues form an $\ell^1$ sequence. Similarly with Hilbert--Schmidt operators, we have
\begin{theorem}[{\cite{fabian2001functional}}]
    Let $A:\mathcal H \to \mathcal K$ be a Hilbert--Schmidt operator. Then $A$ has singular value decomposition $A = U \Sigma V^T$ where $U:\ell^2(\mathbb R)\to \mathcal K$, $V:\ell^2(\mathbb R) \to \mathcal H$ are isometries and $\Sigma = \mathrm{diag}(\sigma_1, \sigma_2, \dots)$ has $(\sigma_i)\in \ell^2(\mathbb R)$. 
\end{theorem}
Thus, Hilbert--Schmidt operators are precisely operators whose singular values form an $\ell^2$ sequence.

\subsection{Sketch Functionals}

We are now ready to define the space of sketch functionals.

\begin{definition}[Centred Sketch Functional]\label{def:sketchfunctional}
    Let $\mathcal H$ be a Hilbert space, and $L^2(\mathbb P)$ be the space of $\mathbb P$ random variables with finite second moments. The space of bounded centred sketch functionals on $\mathcal H$ is defined as
    \begin{align*}
        \mathcal H^{\#}=\bigg\{ \ess \in \mathcal B\big(\mathcal H, L^2(\mathbb P)\big) \text { such that } \mathbb E[\ess(x)]=0 \text { for all } x \in \mathcal H  \bigg\}. 
    \end{align*}
    Equivalently, $\mathcal H^\# = \mathcal B(\mathcal H, L^2_0(\mathbb P))$ where $L^2_0(\mathbb P) = \{X \in L^2(\mathbb P): \mathbb E[X]=0\}$.
\end{definition}

\begin{remark}
    After developing this theory, we discovered that the concept of sketch functionals is closely related to the theory of cylindrical random variables, discussed with more detail in Vakhania--Tarieladze--Chobanyan~\cite{vakhania1987probability}. In fact the space $\mathcal H^{\#}$ is exactly the space of centred cylindrical random variables with finite weak second moments~\cite[\S VI.3.2 and \S VI.5.1]{vakhania1987probability}. We retain our original development of theory as we expect that it will be clearer to the randomised numerical linear algebra community.
\end{remark}

\begin{definition}[Covariance Operator, cf.~{\cite[\S III.2]{vakhania1987probability}}]
\label{def:covariance_operator}
    A symmetric positive semi-definite operator $C\in\mathcal B(\mathcal H)$ is said to be the covariance operator of $\ess \in \mathcal H^\#$ if for all $x,y \in \mathcal H$, $\mathbb E[\ess(x)\ess(y)]=\langle x,Cy\rangle_\mathcal H=x^TCy$.
\end{definition}

\begin{remark}\label{remark:covariance_polarisation}
    Due to the following polarisation argument, to check the covariance operator of $\ess$ is $C$ it is enough to check that $\mathbb E[\ess(x)^2]=x^TCx$. If $\mathbb E[\ess(x)^2]=x^TCx$ for a SPSD operator $C$,
    \begin{align*}
        2\mathbb E[\ess(x)\ess(y)] &= \mathbb E[\ess(x)^2 + 2\ess(x)\ess(y) + \ess(y)^2] - \mathbb E[\ess(x)^2]-\mathbb E[\ess(y)^2]\\
        &= \mathbb E[(\ess(x)+\ess(y))^2] - \mathbb E[\ess(x)^2]-\mathbb E[\ess(y)^2]\\
        &= \mathbb E[\ess(x+y)^2] - \mathbb E[\ess(x)^2]-\mathbb E[\ess(y)^2]\\
        &= (x+y)^TC(x+y) - x^TCx- y^TCy\\
        &= 2x^TCy
    \end{align*}
\end{remark}

\begin{proposition}\label{prop:properties_of_gensketch}
    Let $\ess \in \mathcal H^\#$. Then
    \begin{enumerate}
        \item {\cite[Prop. III.2.1]{vakhania1987probability}} $\ess$ has a SPSD covariance operator $C \in \mathcal B(\mathcal H)$ satisfying $\mathbb E[\ess(x)\ess(y)]=x^TCy$ and $\|C\|_2= \|\ess\|_{\mathcal B(\mathcal H,L^2)}^2$.
        \item (Karhunen--Lo\`eve Property) $\ess$ has covariance operator $C$ with orthonormal eigendecomposition $C = \sum_i \lambda_i p_i p_i^T \iff$ there exists a random sequence of centred $\ess_i$ satisfying $\mathbb E[\ess_i \ess_j]=\delta_{ij}$ such that for any $x \in \mathcal H$, $\ess(x)=\sum\lambda_i^{1/2} \ess_i \langle x,p_i \rangle$ with the convergence of the infinite sum in $L^2(\mathbb P)$. 

    \end{enumerate}
    \begin{proof}\underline{Proof of Claim 1.}\\
        Consider the adjoint of $\ess \in \mathcal B(\mathcal H,L^2(\mathbb P))$, by definition a map $\ess^* \in \mathcal B(L^2(\mathbb P),\mathcal H)$ such that for any $Y \in L^2(\mathbb P)$,
        $$
            \mathbb E[\ess(x)Y]=\langle \ess(x),Y \rangle_{L^2(\mathbb P)}= \langle x, \ess^*(Y)\rangle_{\mathcal H},
        $$
        Let $C= \ess^* \ess$. Then for $x,y\in\mathcal H$,
        $$
            \mathbb E[\ess(x)\ess(y)] = \langle x, \ess^*\ess(y)\rangle_{\mathcal H} = x^T (\ess^*\ess) y
        $$
        and $C$ is clearly bounded and SPSD by its definition. For the second part of the claim, we have that
        \begin{align*} \|\ess\|^2_{\mathcal B(\mathcal H,L^2)}&= \sup_{\|x\|_{\mathcal H}=1} \|\ess(x) \|^2_{L^2(\mathbb P)}\\ &=\sup_{\|x\|_{\mathcal H}=1}\mathbb E[\ess(x)^2] = \sup_{\|x\|_{\mathcal H}=1} x^T Cx=\|C\|_2.
        \end{align*}
            \underline{Proof of Claim 2.}\\
    For any $x \in \mathcal H$ and any choice of orthonormal basis $p_i$,
    \begin{align*}
    \mathbb E\left[\left|\ess(x)-\sum_{i\le n} \ess(p_i) \langle x,p_i\rangle\right|^2\right]&= \mathbb E\left[\left|\ess(x - \sum_{i \le n} p_i \langle x,p_i\rangle)\right|^2\right]\\
    &=  \left(x - \sum_{i \le n} p_i \langle x,p_i\rangle\right)^TC\left(x - \sum_{i \le n} p_i \langle x,p_i\rangle\right) \\
    & \le \|C\|_2 \bigg\|x - \sum_{i \le n} p_i \langle x,p_i\rangle\bigg\|^2_\mathcal H
    \end{align*}
    and since $\{p_i\}$ is an orthonormal basis for $\mathcal H$, we have
    $$
    \bigg\|x - \sum_{i \le n} p_i \langle x,p_i\rangle\bigg\|^2_\mathcal H \to 0
    \text{.}$$
    So,
    $$
    \ess(x)\overset{L^2(\mathbb P)}{=}\sum_i \ess(p_i) \langle x,p_i \rangle \text{,}
    $$

    Now, if $\ess$ has covariance operator $C$, then
        \begin{equation*}
            \sum_i \ess(p_i) \langle x,p_i \rangle= \sum_i \lambda_i^{1/2} \frac{\ess(p_i)}{\lambda_i^{1/2}} \langle  x,p_i \rangle
        \end{equation*}
        and we note that $\ess_i = \ess(p_i)/\lambda_i^{1/2}$ are uncorrelated random variables with unit variance since $\mathbb E[\ess(p_i) \ess(p_j)] = p_i^T C p_j = \lambda_i \delta_{ij}$. \\
        If $\lambda_i=0$, we can take $\ess_i$ to be \emph{any} uncorrelated random variable with unit variance, and since we multiply it by $\lambda_i^{1/2}=0$ it has no effect.
        Conversely assuming the expansion form and making use of Remark~\ref{remark:covariance_polarisation},
        \begin{align*}
            \mathbb E[\ess(x)^2] &= \mathbb E\left [ \left(\sum_i \lambda_i \ess_i^2 \langle x,p_i \rangle^2 \right) + \left( \sum_{i \ne j} \lambda_i^{1/2} \lambda_j^{1/2} \ess_i \ess_j \langle x,p_i \rangle \langle x,p_j \rangle\right)\right ]\\
            &= \mathbb E\left[ \sum_i \lambda_i \ess_i^2 \langle x, p_i \rangle^2 \right] \\
            &= \sum_i \lambda_i \langle x,p_i\rangle^2 = x^TC x.
        \end{align*}
    \end{proof}
\end{proposition}

\begin{remark}\label{remark:interpretations}
    We collect some remarks on notation. For $x \in \mathcal H$, $\ess \in \mathcal H^\#$:
    \begin{itemize}
        \item $\ess(x)$ is a random variable with finite second moments. It is to be interpreted as the $\mathcal H$ inner product of the random vector $\ess$ and the vector $x$. Sometimes we might use quasimatrix notation over $\mathcal H$ so that $\ess(x)=\ess^Tx=x^T\ess$. For this reason, we try to avoid using $\ess x$ which would be the usual notation for a linear map.
        \item $\ess^*$ is the adjoint of $\ess$ as a linear map $\mathcal H\to L^2(\mathbb P)$, which acts on elements of $L^2(\mathbb P)$ with range in $\mathcal H$. By definition, for a random variable $Y\in L^2(\mathbb P)$, $\ess^*(Y)=\tilde y$ has $\langle x, \tilde y\rangle_{\mathcal H}=\mathbb E[\ess(x)Y]$. $\ess^*Y$ can be interpreted as the covariance vector between the random vector $\ess$ and the random variable $Y$.
    \end{itemize}
\end{remark}

\begin{definition}[Isotropy]\label{def:isotropy}
    A sketch functional $\ess\in\mathcal H^\#$ is said to be isotropic if its covariance operator is the identity $I_\mathcal H$.
\end{definition}

\begin{example}[Standard Gaussian]
    Let $p_i$ be an orthonormal basis for $\mathcal H$, and sample $g_i\overset{iid}\sim \mathcal N(0,1)$. Then $\ess(x) = \sum g_i \langle x,p_i\rangle$ is a centred isotropic sketch functional with the property that $\ess(x) \sim \mathcal N(0,\|x\|^2_{\mathcal H})$ for all $x \in \mathcal H$. This is referred to as an isonormal Gaussian process in probability literature, for example Nualart~\cite{nualart2006malliavin}.
\end{example}

Proposition~\ref{prop:whitenoiseproblem} shows that any isotropic sketch functional has infinite expected square $\mathcal H$-norm. Due to the following example, the stronger statement that isotropic sketch functionals have infinite $\mathcal H$-norm almost surely is false.

\begin{example}[Almost-Surely Finite Isotropic Sketch {\cite[Theorem III.2.2 and Cor. 2]{vakhania1987probability}}]\label{example:asfinite}
Let $p_i$ be an orthonormal basis for $\mathcal H$ a separable Hilbert space. Let $\ess$ take a value in the direction of $p_i$ with probability $2^{-i}$. To ensure that $\ess$ is centred, it takes the value $c_ip_i$ and $-c_ip_i$ both with probability $2^{-(i+1)}$. If we can choose $c_i<\infty$ to make the sketch isotropic, we are done. $\mathbb E[\langle \ess,p_i\rangle \langle \ess,p_j\rangle]=0$ for $i \ne j$ by construction, as one of the factors will be zero, so we only need that $\mathbb E[\langle \ess, p_i\rangle ^2]=1$ for all $i$. $\mathbb E[\langle \ess,p_i\rangle^2] = 2^{-i}c_i^2$, so if $c_i = 2^{i/2}$ we are done.
\end{example}

\begin{definition}[Images of Sketch Functionals]\label{def:product}
   Let $\ess$ be a sketch functional on $\mathcal{H}$. For a bounded operator $A \in \mathcal{B}(\mathcal{H},\mathcal{K})$, where $\mathcal{K}$ is another separable Hilbert space, define the product $A\ess$ as the sketch functional defined by $A\ess(x)=\ess(A^Tx)$.
\end{definition}

As noted in Remark~\ref{remark:interpretations}, we want $\ess(x)$ to act like an inner product between $\ess$ and $x$. Definition~\ref{def:product} was informed by the equality $\langle A\ess, x\rangle = \langle \ess, A^Tx\rangle$.

\begin{proposition}[{\cite[\S III.2.1, property (e)]{vakhania1987probability}}] \label{prop:product_cov}
    If $\ess \in \mathcal H^\#$ has covariance operator $C$, then $A\ess$ has covariance operator $ACA^T$.
    \begin{proof}
        By the proof of Proposition~\ref{prop:properties_of_gensketch}, the covariance operator of $A\ess$ is $(A\ess)^*(A\ess)$. Since $A\ess(x)=\ess (A^Tx)$, as operators we can write $A\ess = \ess A^T$ and
        $$
        (A\ess)^*(A\ess)=(\ess A^T)^* (\ess A^T)= A \ess^* \ess A^T = ACA^T. 
        $$
    \end{proof}
\end{proposition}

In the case of~\cite{boulle2021generalization}, the covariance operator $C$ is trace class. The following proposition is the reason this approach circumvents the white noise problem.
\begin{proposition}[{\cite[Prop. II.2.5]{vakhania1987probability}}]\label{prop:traceclass}
If $C$ is trace class, so $(\lambda_i)_i \in \ell^1$, then $\ess$ exists as an element of $\mathcal H$ almost surely. That is, there is an $\mathcal H$-valued $\hat \ess$ with $\mathbb E [\| \hat \ess\|^2_{\mathcal H}]= \tr(C)$ such that for each $x$, $\ess(x)= \langle \hat \ess, x\rangle$ almost surely.
    \begin{proof}
        In this case, $\ess(x) = \sum_i \lambda_i^{1/2} \ess_i\langle x,p_i \rangle$ arises as the inner product of $x$ with $\sum_i \lambda_i^{1/2} \ess_i p_i$. This has finite expected norm:
        \begin{align*}
            \mathbb E\left [\left \| \sum_i \lambda_i^{1/2} \ess_i p_i\right \|_{\mathcal{H}}^2\right] &= \sum_i \lambda_i \mathbb E[\ess_i^2] = \sum_i  \lambda_i <\infty
        \end{align*}
    and so exists as an element of $\mathcal H$ (has finite $\mathcal H$ norm) almost surely.
    \end{proof}
\end{proposition}

Kressner et al.~\cite{KRESSNER2026259} directly tackle the white noise problem for Gaussian processes; the following demonstrates that it generalises to the sketch functionals of this work.
\begin{corollary}
    If $A \in \text{HS}(\mathcal H,\mathcal K)$, then $A \ess$ exists as an element of $\mathcal K$ almost surely.
    \begin{proof}
        According to Proposition~\ref{prop:product_cov} the covariance operator is $ACA^T$, which is a product of two $HS$ operators ($AC$ and $A^T$) and hence trace class.
    \end{proof}
\end{corollary}

\begin{definition}[Sketch Functional Quasimatrices]\label{definition:generalised_sketch_matrices}
    Call a collection of $k$ sketch functionals on $\mathcal{H}$, $\Ess = \big ( \ess_1, \ess_2, \dots, \ess_k\big )$ a sketch functional quasimatrix. For $A \in \mathcal{B}(\mathcal{H},\mathcal{K})$, we write the product $A \Ess = (A \ess_1, A\ess_2, \dots, A\ess_k)$. For $B \in \mathbb{R}^{k\times \ell}$, we write $\Ess B = \left( \sum_{j=1}^k B_{j1}\,\ess_j, \ldots, \sum_{j=1}^k B_{j\ell}\,\ess_j \right)$.
\end{definition}
It is readily checked that the products defined above are associative. We have followed Townsend--Trefethen~\cite{townsend2015continuous} in calling this a ``quasimatrix'', a matrix-like object with infinite-dimensional columns.

\begin{remark}\label{remark:finite_dim_collapse}
    The product $A\Ess$ is worth dwelling on when $A \in \mathcal B(\mathcal H, \mathbb R^k)$ for some finite $k$, as will shortly be the case with $A=V_k^T$. Then each $A\ess_j$ is a sketch functional on the finite-dimensional space $\mathbb R^k$, that is, an element of $(\mathbb R^k)^\#$, and by the (finite-dimensional) Riesz representation theorem $A\ess_j$ is simply an ordinary random vector in $\mathbb R^k$: the one representing the functional $y \mapsto A\ess_j(y)$. So $A\Ess=(A\ess_1,\dots,A\ess_\ell)$, a quasimatrix whose columns are these random vectors, is nothing more exotic than an ordinary $k\times\ell$ random matrix. It is this collapse to a finite-dimensional object that lets us apply ordinary matrix notions --- rank, singular values, the Moore--Penrose pseudoinverse --- directly to quasimatrix products landing in a finite-dimensional space, as we do with $\Ess_1$ from Lemma~\ref{lemma:inf_dim_structural_bound} onwards.
\end{remark}

Now that we have a useful notion of products with white noise, we can compute a randomised SVD and give the infinite-dimensional form of the structural bound that we require.

\begin{algorithm} 
\caption{Infinite-Dimensional Randomised SVD Algorithm (General Covariance)}
\label{alg:boullealg}
\textbf{Input:} Hilbert--Schmidt target operator $A \in \text{HS}(\mathcal H,\mathcal{K})$, sketch size $\ell$\\
\textbf{Output:} rank-$\ell$ approximation $\hat A$
\begin{algorithmic}[1]
\State Sample independent sketch functionals $\Ess = (\ess_1, \dots, \ess_\ell)$
\State $Y \gets A \Ess$ (according to Definition~\ref{definition:generalised_sketch_matrices})
\State $Q \gets \text{orth}(Y)$
\State $\tilde U, \tilde \Sigma, \tilde V \gets \mathrm{svd}(Q^TA)$
\State $\hat A \gets Q Q^T A=(Q\tilde U) \tilde \Sigma \tilde V^T$
\end{algorithmic}
\end{algorithm} 

\begin{lemma} \label{lemma:inf_dim_structural_bound}
    Let $A:\mathcal{H} \to \mathcal{K}$ be Hilbert--Schmidt with SVD $U \Sigma V^T$. With sketch functional quasimatrix $\Ess = \big ( \ess_1, \dots ,\ess_\ell \big )$ and target rank $k \le \ell$, define 
    \begin{align*}
    \Sigma_\perp &= \mathrm{diag}(\sigma_{k+1}, \sigma_{k+2}, \dots)\\
    V_k &= (v_1, \dots, v_k), \quad V_\perp = (v_{k+1},v_{k+2},\dots)\\
    \Ess_1 &=  V_k^T \Ess, \quad\Ess_2 = V_\perp^T\Ess
\end{align*}
as sketch functional products. Assuming $\Ess_1$ is full rank,
\begin{equation*}
    \left \|(I - QQ^T)A\right \|^2 \le \| \Sigma_\perp \|^2+\|\Sigma_\perp \Ess_2 \Ess_1^\dagger\|^2
    \end{equation*}
in both the spectral and Hilbert--Schmidt norms, with $Q=\mathrm{orth}(A \Ess)$.
\end{lemma}

\begin{proof}
   We follow the definitions in Definition \ref{definition:generalised_sketch_matrices}. Since $QQ^T$ is the orthogonal projector onto the columns of $A \Ess$, we have that $(I-QQ^T)A\Ess(V_k^T \Ess)^\dagger V_k^T=0$. Therefore, 
    \begin{equation*}
        (I - Q Q^T)A = (I - Q Q^T)(A - A \Ess (V_k^T \Ess)^\dagger V_k^T)
    \end{equation*}
Now,
\begin{align*}
    A- A \Ess (V_k^T \Ess)^\dagger V_k^T &= A - (U_k \Sigma_k V_k^T \Ess(V_k^T \Ess)^\dagger V_k^T+ U_\perp \Sigma_\perp V_\perp ^T \Ess(V_k^T \Ess)^\dagger V_k^T)\\
    &= A - U_k \Sigma_kV_k^T - U_\perp \Sigma_\perp V_\perp ^T \Ess(V_k^T \Ess)^\dagger V_k^T\\
    &= U_\perp \Sigma_\perp V_\perp ^T- U_\perp \Sigma_\perp (V_\perp^T \Ess)(V_k^T \Ess)^\dagger V_k^T\text{,}
\end{align*}
with the second equality following because $\Ess_1$ is a full-rank $k\times \ell$ matrix.
Hence, in either norm,
\begin{align*}
    \|(I-QQ^T)A\|^2 &= \|(I-QQ^T)(A-A \Ess(V_k^T \Ess)^\dagger V_k^T)\|^2\\
    &= \|(I-QQ^T)(U_\perp \Sigma_\perp V_\perp^T - U_\perp \Sigma_\perp \Ess_2 \Ess_1^\dagger V_k^T)\|^2 \\
    & \le \|U_\perp \Sigma_\perp V_\perp^T - U_\perp \Sigma_\perp \Ess_2 \Ess_1^\dagger V_k^T\|^2\\
    &\le \|U_\perp \Sigma_\perp V_\perp^T\|^2 + \|U_\perp \Sigma_\perp \Ess_2 \Ess_1^\dagger V_k^T\|^2\\
    &= \|\Sigma_\perp\|^2 + \|\Sigma_\perp \Ess_2 \Ess_1^\dagger\|^2\text{.}
\end{align*}
where the penultimate inequality follows from orthogonality of the ranges $V_\perp$ and $V_k$.
\end{proof}

\begin{remark}
    The full-row-rank assumption for $\Ess_1 \in \mathbb R^{k\times \ell}$ depends on the sketch distribution. For independent centred Gaussian sketch functionals with covariance operator $C$, $\Ess_1$ has full row rank almost surely whenever $\ell\ge k$ and $V_k^TCV_k$ is positive definite. If this projected covariance is singular, full row rank is impossible. For discrete projected sketches, such as a Rademacher sketch in a coordinate subspace, rank deficiency can occur with positive probability. Repeated columns alone do not force rank deficiency when $\ell>k$: for example, the columns $(1,1)^T$, $(1,1)^T$ and $(1,-1)^T$ still span $\mathbb R^2$. For independent isotropic sub-Gaussian sketches, Theorem~\ref{theorem:twosided} gives a quantitative bound on rank-failure probability when its lower singular-value estimate is positive, as used in Section~\ref{subsection:mainresults}.
\end{remark}

\begin{remark}\label{rem:complex}
    The definitions of this section extend naturally to the complex case; $\mathcal H$ is instead a complex Hilbert space and $\mathcal H^{\#}=\mathcal B(\mathcal H, L^2_0(\mathbb P; \mathbb C))$. Replace transpose by conjugate transpose throughout and much remains unchanged. We will briefly consider such a case in Section~\ref{subsection:discrete_schrod}.
\end{remark}

\subsection{Reduction to Finite Dimensions}
In theory, storing a sample of the sketch functional like the isotropic Gaussian exactly on a computer is impossible as it relies on infinitely many components $\ess_i$ which do not shrink. Following common computational practice, we project $\ess = \sum_{i \in \mathbb N} \ess_i \langle  \cdot, p_i\rangle$ onto the first $n$ basis vectors $p_i$ and instead treat $\hat \ess = \sum_{i \le n} \ess_i \langle \cdot, p_i\rangle$ (as is done, for example, by Chebfun~\cite{driscoll2014chebfun}). If $P_n = (p_1, p_2, \cdots, p_n)$ then we can write this projection as a sketch functional product $P_nP_n^T \ess$ (see Definition \ref{def:product}). For every fixed $x\in\mathcal H$, $(\ess-P_nP_n^T\ess)(x)\to0$ in $L^2(\mathbb P)$, since $P_nP_n^Tx\to x$ in $\mathcal H$ and $\ess$ is bounded. If $C$ is not compact, however, the truncations do not converge to $\ess$ in the operator norm of $\mathcal B(\mathcal H,L^2(\mathbb P))$. Further, if $C$ is not trace-class, the residual $\ess-P_nP_n^T\ess$ has infinite expected squared $\mathcal H$-norm. The following Lemma shows that projecting does work when combined with a Hilbert--Schmidt operator.

\begin{lemma}
    If $\ess$ is a sketch functional on $\mathcal{H}$ with bounded covariance operator $C$, and $A \in HS(\mathcal{H},\mathcal{K})$, then
    \begin{equation}
    \mathbb E[\|A\ess - AP_n P_n^T \ess\|_{\mathcal{K}}^2] = \|A(I-P_nP_n^T)C^{1/2}\|_{HS}^2 \to 0
    \end{equation}
    as $n \to \infty$.
    \begin{proof}
        This follows from
        \begin{equation*}
            A \ess - AP_nP_n^T \ess = A(I-P_nP_n^T)\ess,
        \end{equation*}
        which has covariance operator $A(I-P_nP_n^T)C(I-P_nP_n^T)A^T$
        so
        \begin{align*}
            \mathbb E[\|A\ess - AP_n P_n^T \ess\|_{\mathcal{K}}^2] &= \tr(A(I-P_nP_n^T)C(I-P_nP_n^T)A^T)\\
            &= \|A(I-P_nP_n^T)C^{1/2}\|_{HS}^2.
        \end{align*}
        To prove the limit, note that 
        \begin{align*}
            \|A(I-P_nP_n^T)C^{1/2}\|_{HS}^2 &\leq \|A(I-P_nP_n^T)\|_{HS}^2 \|C^{1/2}\|_{2}^2 \\
             &\leq \|C\|_2\sum_{j > n} \|A p_j\|_2^2 \\
             & \to 0
        \end{align*}
        since $A$ is Hilbert--Schmidt so $\|A\|_{HS}^2 = \sum _{n \in \mathbb N} \|Ap_n\|_2^2 < \infty$.
    \end{proof}
\end{lemma}

In most applications, this will decrease sufficiently quickly so as to not be an issue.

We also analyse the error this truncation provokes in the randomised SVD. When applying the infinite-dimensional randomised SVD to $A$ with a truncated sketch, we form $Y = A P_nP_n^T \Ess$ and set $Q = \mathrm{orth}(Y)$. The resulting $Q$ is equivalent to what we would get from applying the algorithm to $AP_nP_n^T$ with a full sketch, and usual theory yields guarantees on $\|AP_nP_n^T - QQ^TAP_nP_n^T\|$. We then have
\begin{align*}
    \left \| A - QQ^TA \right \|_{HS}^2 &= \left \| (A P_nP_n^T + AP_\perp P_\perp^T) - (QQ^T AP_nP_n^T + QQ^T A P_\perp P_\perp^T)\right \|_{HS}^2 \\
    &= \left \| A P_nP_n^T - QQ^T A P_nP_n^T \right \|_{HS}^2 + \left \| A P_\perp P_\perp^T - QQ^T AP_\perp P_\perp^T\right \|_{HS}^2\\
    & \le \|A P_nP_n^T - QQ^T AP_n P_n^T\|_{HS}^2 + \left \|AP_\perp P_\perp^T \right \|_{HS}^2\left \| I-QQ^T\right \|_2^2 \\
    &=\|A P_nP_n^T - QQ^T AP_n P_n^T\|_{HS}^2 + \left \|AP_\perp P_\perp^T \right \|_{HS}^2,
\end{align*}
so that the error induced by this truncation can be no greater than $\| AP_\perp P_\perp^T\|_{HS}^2$. As detailed in the proof of Corollary~\ref{corollary:nystrom}, the error of a Nystr\"om approximation can be rephrased in terms of the error of a randomised SVD, so this discussion also carries over.

\subsection{A Note on the Gaussian Case}

One common definition of a Gaussian random vector $\ess \in \mathbb R^n$ is the requirement that $\langle \ess, x\rangle_{\mathbb R^n}$ is a Gaussian random variable for all $x \in \mathbb R^n$. This seems natural with sketch functional notation, and we follow that convention here and later in the sub-Gaussian portion of the paper.

\begin{definition}[Gaussian Sketch Functional {\cite[\S VI.3.1]{vakhania1987probability}}]\label{def:gaussian_skfun}
    We say a sketch functional $\ess$ on $\mathcal H$ is Gaussian if for all $x \in \mathcal H$, $\ess(x)$ is a Gaussian random variable. By linearity, a finite collection $\big( \ess(x_1), \ess(x_2), \cdots, \ess(x_n)\big)$ is also jointly Gaussian.
\end{definition}

The approach of Kressner et al.~\cite{KRESSNER2026259} yields the following Theorem when sketch functional notation is used. Implicitly, they sampled a centred isotropic Gaussian sketch functional.

\begin{theorem}[\cite{KRESSNER2026259}]\label{theorem:KPU_maintheorem}
Let $p\ge 4$ be a small oversampling parameter, and sample $\ell=k+p$ centred isotropic Gaussian sketch functionals.
If $\hat A$ is the output of Algorithm~\ref{alg:boullealg},
    \begin{equation}
        \mathbb E \left [\left \| A- \hat A\right\|^2_{HS}\right] \le \bigg(1+\frac{k}{p-1}\bigg)\big\| \Sigma_\perp\big\|_{HS}^2,
    \end{equation}
    and for any $u,t \ge 1$, the following inequality holds with probability at least $1-2t^{-p}-e^{-u^2/2}$
    \begin{equation}
    \left\|A - \hat A\right\|_{HS} \le \big\| \Sigma_\perp\big\|_{HS}+t\sqrt{\frac{3k}{p+1}}\big\| \Sigma_\perp\big\|_{HS} + ut\frac{e\sqrt{k+p}}{p+1}\big\| \Sigma_\perp\big\|_2
    \end{equation} 
\end{theorem}

We also give the following result for centred Gaussian sketches with bounded covariance operators.

\begin{theorem}\label{theorem:gaussian_cov}
    Let $C$ be a bounded SPSD operator on $\mathcal H$. Let $p \ge 4$ be a small oversampling parameter and sample $\ell=k+p$ independent centred Gaussian sketch functionals with covariance operator $C$. For a Hilbert--Schmidt target operator $A= U \Sigma V^T$, assume $V_k^TC^{1/2}$ has rank $k$. If $\|\Sigma_\perp\|_{HS}>0$, denote
    $$\beta_k= \frac{\| \Sigma_\perp V_\perp^TC^{1/2}\|_{HS}}{\sigma_k(V_k^TC^{1/2})\| \Sigma_\perp\|_{HS}}\text{.}$$
    Then with probability at least $1-(1+u)^{\ell /2}\exp(- u \ell /2)-t^{-(p+1)}$,
    \begin{equation}
        \|(I-QQ^T)A\|_{HS} \le \|\Sigma_\perp\|_{HS}\left(1 +  \beta_k\frac{e(k+p)}{p+1}t \sqrt{1+u}\right).
    \end{equation}
    If $\|\Sigma_\perp\|_{HS}=0$, then $QQ^TA=A$ almost surely, without defining $\beta_k$.
\end{theorem}

\begin{proof}
If $\|\Sigma_\perp\|_{HS}=0$, then $\Ess_1$ has full row rank almost surely because its Gaussian columns have positive definite covariance $V_k^TCV_k$. Lemma~\ref{lemma:inf_dim_structural_bound} therefore gives $QQ^TA=A$. It remains to consider $\|\Sigma_\perp\|_{HS}>0$.

It is sufficient to bound $\| \Sigma_\perp \Ess_2 \Ess_1^\dagger\|_{HS} \le \|\Sigma _\perp \Ess_2 \|_{HS} \|\Ess_1^\dagger\|_2$. Note that if $\Omega$ is an isotropic Gaussian sketch, then $\Ess \overset{d}{=} C^{1/2} \Omega$.

\underline{$\Sigma_\perp \Ess_2$:}
By \cite[Lemma 7]{boulle2021generalization}, we have
\begin{equation*}
    \mathbb P\big( \| \Sigma_\perp \Ess_2\|_{HS}^2 > \ell(1+u)\tr(\Sigma_\perp V_\perp^T C V_\perp\Sigma_\perp)\big) \le  (1+u)^{\ell/2}\exp(-u\ell/2)
\end{equation*}
Note that $\tr(\Sigma_\perp V_\perp^T C V_\perp\Sigma_\perp) = \|\Sigma_\perp V_\perp^TC^{1/2}\|_{HS}^2$. If this trace is zero, then $\Sigma_\perp\Ess_2=0$ almost surely and the strict upper-tail event has probability zero.

\underline{$\Ess_1^\dagger$:}
We have that $\Ess_1 = V_k^T C^{1/2} \Omega = V_k^TC^{1/2} W_kW_k^T \Omega$, so $$\sigma_k (\Ess_1) \ge \sigma_k(V_k^TC^{1/2} W_k) \sigma_k(W_k^T \Omega) = \sigma_k(V_k^TC^{1/2}) \sigma_k(W_k^T \Omega).$$
$W_k^T\Omega$ is a $k \times \ell$ standard Gaussian matrix, and hence is subject to \cite[Proposition 10.4]{halko2011finding} yielding

\begin{equation*}
    \mathbb P\left(\|(W^T_k \Omega)^\dagger\|_2 \ge \frac{e\sqrt{k+p}}{p+1}t\right) \le t^{-(p+1)}
\end{equation*}
therefore
\begin{align*}
    \mathbb P\left( \| \Ess_1^\dagger\|_2 \ge \frac{e\sqrt{k+p}}{(p+1)\sigma_k(V_k^TC^{1/2})}t\right) &\le \mathbb P\left(\frac{\| (W_k^T\Omega)^\dagger\|_2}{\sigma_k(V_k^TC^{1/2})} \ge\frac{e\sqrt{k+p}}{(p+1)\sigma_k(V_k^TC^{1/2})}t\right)  \\
    &= \mathbb P\left( \|(W_k^T \Omega) ^\dagger \|_2 \ge \frac{e\sqrt{k+p}}{(p+1)}t \right) \\ 
    &\le t^{-(p+1)}
\end{align*}
By a union bound, the probability that these events hold simultaneously is at least $1-(1+u)^{\ell /2}\exp(- u \ell /2)-t^{-(p+1)}$ and the claim follows.
\end{proof}

\begin{remark}\label{rem:BoulleTownsendfactors}
    Let us examine the scaling factor $\beta_k$. If $\lambda_{\max}(V_\perp^TCV_\perp)=0$ but $\lambda_{\min}(V_k^TCV_k)\ne 0$, then the columns of $C$ (and therefore of $\Ess$ with sufficient $\ell$) span $V_k$, the correct first $k$ singular vectors. This is a happy breakdown; we have chosen the correct subspace, and so the algorithm will converge immediately to the optimal rank-$k$ approximation of the target operator. On the other hand, if $\lambda_{\max}(V_\perp^TCV_\perp)$ is much larger than $\lambda_{\min}(V_k^TCV_k)$, then our sketch functionals will have relatively large components in the wrong subspace $V_\perp$, and the approximation will become worse.
\end{remark}

For comparison, we state here the main result of Boull\'e--Townsend~\cite{boulle2021generalization}.

\begin{theorem}[\cite{boulle2021generalization}]\label{theorem:bt_maintheorem}
Let $C$ be a trace-class SPSD operator on $\mathcal H$, let $k\in\mathbb N$ be the target rank and let $p\ge4$. Sample $\ell=k+p$ independent centred Gaussian random vectors in $\mathcal H$ with covariance operator $C$. Let the Hilbert--Schmidt target operator $A$ have SVD $U\Sigma V^T$ and assume $V_k^TCV_k$ is positive definite. If $\|\Sigma_\perp\|_{HS}>0$, then for all $u,t \ge 1$, with probability at least $1-t^{-p} - [ue^{-(u^2-1)/2}]^{k+p}$,
    \begin{equation}
        \|(I-QQ^T)A\|_{HS} \le \left( 1+ut\sqrt{(k+p)\frac{3k}{p+1} \frac{\alpha_k}{\gamma_k}} \right) \| \Sigma_\perp\|_{HS}{}
    \end{equation}
    where $\alpha_k = \tr(\Sigma_\perp V_\perp^T C V_\perp\Sigma_\perp)/(\|C\|_2 \|\Sigma_\perp\|_{HS}^2)$ and $\gamma_k = k/(\| C\|_2\tr((V_k^TCV_k)^{-1}))$.
    If $\|\Sigma_\perp\|_{HS}=0$, then $QQ^TA=A$ almost surely, without defining $\alpha_k$.
\end{theorem}

\begin{remark}
Here, the factor $\alpha_k/\gamma_k$ is responsible for measuring the alignment of the covariance operator to the target operator. For us, $\beta_k$ plays this role, and they behave similarly. Theorem~\ref{theorem:bt_maintheorem} remains valid for bounded SPSD $C$ if the Gaussian random vectors are replaced by independent centred Gaussian sketch functionals with covariance operator $C$, retaining the other hypotheses. Boundedness of $C$ and the Hilbert--Schmidt property of $A$ ensure that the tail trace in $\alpha_k$ is finite. The inverse in $\gamma_k$ additionally requires $V_k^TCV_k$ to be positive definite, and the normalised factor $\alpha_k$ requires $\|\Sigma_\perp\|_{HS}>0$. Under these conditions, both factors are well-defined and finite, with $\gamma_k>0$; the zero-tail case instead gives exact recovery almost surely.
\end{remark}

\section{Sub-Gaussian Random Processes}
The main appeal of sketch functionals is their ability to capture more general distributions --- we only make assumptions on the first and second moments. In this section, we turn our attention to sub-Gaussian distributions, demonstrating the theory in a non-Gaussian setting. We begin with a brief introduction to the theory necessary for the rest of the paper. Once the necessary definitions are in place, we state our main results on low-rank operator approximation in Section~\ref{subsection:mainresults}.

\subsection{Sub-Gaussian Random Variables, Vectors and Processes}

The following theory around sub-Gaussian random variables will be important. We will define that term and note some useful results. These can be found with more detail in~\cite{vershynin2018high}.

\begin{definition}
    A real-valued random variable $X$ taking values in $\mathbb R$ is called sub-Gaussian if there is a $K_X>0$ such that $\mathbb P(|X|\ge t) \le 2 \exp(-K_Xt^2)$ for all $t \ge 0$.
\end{definition}

\begin{proposition}\label{prop:subg_equivalences}
    Let $X$ be a random variable taking values in $\mathbb R$. The following are equivalent, with parameters $K_i$ differing by at most an absolute constant factor.\\
    \begin{enumerate}
        \item Tails: $\exists K_1>0$ so that for all $t\ge 0$
        \begin{equation}
            \mathbb P(|X| \ge t) \le 2\exp(-t^2/K_1^2) \label{property:tails}
        \end{equation}
        \item Moments: $\exists K_2>0$ so that for all $p \ge 1$
        \begin{equation}
            \|X\|_{L^p} \le K_2 \sqrt p \label{property:moments}
        \end{equation}
        \item Moment Generating Function of $X^2$: $\exists K_3>0$ so that 
        \begin{equation}
            \mathbb E[ \exp(X^2/K_3^2)] \le 2 \label{property:mgfsq}
        \end{equation}
\end{enumerate}
    and if $\mathbb E[X]=0$, these are all equivalent to
    \begin{enumerate}[resume]
        \item Moment Generating Function of $X$: $\exists K_4>0$ so that for all $\lambda \in \mathbb R$
        \begin{equation}
            \mathbb E[\exp(\lambda X)] \le \exp(K_4^2 \lambda^2) \label{property:mgf}
        \end{equation}
    \end{enumerate}
\end{proposition}

\begin{remark}
    For centred $X$ and $K_i$ the optimal constants, we remark that

\begin{equation}
    K_1 \le 2 K_4 \le 2 \sqrt{\frac32} K_3 \le 4 \sqrt{\frac{3e}{2}} K_2 \le 12 \sqrt{\frac{3e}{2}} K_1
\end{equation}

\end{remark}

\begin{definition}[Sub-Gaussian Norm] \label{def:sub-Gaussian}
Define the sub-Gaussian norm $\|X\|_{\psi_2}$ of a centred sub-Gaussian random variable $X$ according to moment generating function,
\begin{equation*}
    \|X\|_{\psi_2} = \inf\left \{ K > 0 : \mathbb E[\exp(\lambda X)] \le \exp\left(\frac{K^2\lambda^2}{2} \right) \text{ for all } \lambda \in \mathbb{R}\right \}.
\end{equation*}
If $X$ is not centred then define $\|X\|_{\psi_2}^2 = \|X - \mathbb{E}[X]\|_{\psi_2}^2 + |\mathbb{E}[X]|^2$.
\end{definition}

For centred $X$, the inequality $\mathbb{E}[\mathrm{e}^{\lambda X}] \leq \mathrm{e}^{K^2\lambda^2/2}$ implies $1 + \frac{\lambda^2}{2}\mathbb{E}[X^2] + o(\lambda^2) \leq 1 + \frac{\lambda^2}{2}K^2 + o(\lambda^2)$, which, after dividing by $\lambda^2$ and sending $\lambda$ to zero implies
\begin{equation}\label{eqn:variancebound}
    \mathbb{E}[X^2] \leq K^2.
\end{equation}
For this reason, this form of sub-Gaussian norm is sometimes known as the \emph{variance proxy}.

\begin{proposition}
    Let $X \sim N(0,\sigma^2)$, let $Y \sim \mathrm{Rademacher}$ and let $Z$ with $\mathbb E[Z]=0$ take values in $[-a,a]$ almost surely. Then
    \begin{equation}
        \|X\|_{\psi_2} = \sigma, \quad  \|Y\|_{\psi_2} = 1,  \quad \text{ and } \quad \|Z\|_{\psi_2} \leq a.
    \end{equation}
\end{proposition}

\begin{proposition}\label{prop:pythag} Suppose $X_i$ are independent and centred sub-Gaussian random variables. Then for a countable index set $I$, 
    \begin{equation}
        \left \|\sum_{i\in I} X_i\right \|_{\psi_2}^2 \le \sum_{i\in I} \|X_i\|_{\psi_2}^2.
    \end{equation}
\end{proposition}
If we followed a different property in Proposition~\ref{prop:subg_equivalences} to define a sub-Gaussian random variable, then Proposition~\ref{prop:pythag} would only be true up to a constant multiple.

\begin{definition}[Sub-Gaussian Vectors and Sketch Functionals]\label{def:subg-vec-sketches} A sub-Gaussian vector is a random vector $\omega$ such that $\|\langle \omega, x \rangle\|_{\psi_2} \le K$ for all $\|x\|=1$.
    The sub-Gaussian vector norm is the infimum over all $K$ satisfying the above. Similarly, a sub-Gaussian sketch functional $\ess$ is one such that $\|\ess(x)\|_{\psi_2} \le K$ for all $\|x\|=1$, and the sub-Gaussian norm is the infimum over all $K$ satisfying the above.
\end{definition}
\begin{proposition}
    Take $\ess_i$ to be independent random variables with sub-Gaussian norm $K$ and unit variance. Then for any choice of orthonormal basis $\{p_i\}$, the sketch functional defined by $\ess(x) = \sum_i \ess_i \langle x, p_i\rangle$ is an isotropic sub-Gaussian sketch functional with sub-Gaussian norm at most $K$ as in Definition \ref{def:subg-vec-sketches}.
    \begin{proof}
       Since each $\ess_i$ is centred, $\ess(x)$ is centred for any $x$ and
       \begin{align*}
           \mathbb E[\ess(x)^2] &= \mathbb E\left[ \left(\sum_i \ess_i \langle x,p_i \rangle \right)^2 \right] \\ 
           &= \left \|  \sum_i \ess_i \langle x,p_i \rangle \right \|^2_{L^2(\mathbb P)}\\
           &= \sum_i \langle x,p_i \rangle^2 = \|x\|^2_{\mathcal H},
       \end{align*}
       where the third equality is because the $\ess_i$ form an orthonormal set in the Hilbert space $L^2(\mathbb P)$ and the fourth is because the $p_i$ form an orthonormal basis for $\mathcal H$. Thus $\mathbb E[\ess(x)^2]= x^T I_{\mathcal H}x$ and $\ess$ is isotropic. 
        For the sub-Gaussian norm, with any $\|x\| = 1$:
        \begin{align*}
            \| \ess(x)\|_{\psi_2}^2 &= \left \| \sum_i   \ess_i \langle x,p_i\rangle\right \|_{\psi_2}^2\\
            & \le \sum_i \big \| \ess_i \langle x, p_i\rangle\big \|_{\psi_2}^2\\
            &= \sum_i \langle x,p_i\rangle^2 \|\ess_i\|_{\psi_2}^2\\
            &=K^2 \sum_i \langle x,p_i \rangle^2 = K^2 \text{.}
        \end{align*}
        The inequality is by Proposition \ref{prop:pythag}.
    \end{proof}
\end{proposition}

\begin{example}[Isotropic Rademacher Sketch]
    Let $\ess_i \overset{iid}{\sim} U(\{-1,+1\})$. Let $p_i$ be an orthonormal basis for $\mathcal H$ and let $\ess \in \mathcal H^\#$ be defined by $\ess(x)= \sum_i \ess_i \langle x,p_i\rangle_{\mathcal H}$. Then $\ess$ is a sub-Gaussian sketch functional with sub-Gaussian norm $1$, since $\|\ess_i\|_{\psi_2}=1$. 
\end{example}

\subsection{Sub-Gaussian Processes}

A Gaussian process is usually defined as a stochastic process $X$ such that for all $t_1, \cdots, t_n$, $\big(X(t_1),\cdots, X(t_n)\big)$ is jointly Gaussian. Despite this looking different to Definition~\ref{def:gaussian_skfun}, it is common to refer to infinite-dimensional Gaussian random vectors as Gaussian processes. This is partly due to the Karhunen--Lo\`{e}ve theorem, which shows that over $L^2(I)$ the terms are equivalent for trace-class covariance operators.

\begin{theorem}[\cite{alexanderian2015brief}]
        Let $c(t,s)$ be a continuous covariance kernel on $I \times I$ for a compact interval $I \subset \mathbb R$ such that the operator $Cf(s) = \int f(t)c(t,s)\,\mathrm dt$ is a trace-class SPSD operator on $L^2(I)$. Let $C = \sum \lambda_ip_ip_i^T$ be an orthonormal eigendecomposition. Then $X$ is a mean-square continuous Gaussian process with covariance kernel $c$ and continuous expectation $\mu(t) = \mathbb E[X(t)]$ if and only if $X$ has the representation $X = \mu + \sum \lambda_i^{1/2} x_ip_i$ where $x_i \overset{iid} \sim \mathcal N(0,1)$, and the series converging in $L^2(\mathbb P)$ uniformly in $t$.
\end{theorem}

To be clear, if $X$ is continuous, then this gives an explicit basis expansion for $X$ as a Gaussian random vector. In the other direction, an explicit basis expansion determines the pointwise distribution of $X(t)$ which is also Gaussian. This raises a question whether there is such an equivalence in the sub-Gaussian case. In particular, are random processes which are sub-Gaussian pointwise also sub-Gaussian random vectors in the sense of Definition~\ref{def:subg-vec-sketches}? What about the converse?

\begin{proposition}
    Let $I=[a,b]$ be a finite interval and $\omega$ be a centred stochastic process taking values in $L^2(I)$ which is continuous almost surely and $\|\omega(x)\|_{\psi_2} \le K$ for any $x \in I$. Then for any $f \in L^2(I)$, $ \|\langle \omega, f\rangle\|_{\psi_2} \le K \sqrt{b-a}\|f\|_{L^2}$.
\end{proposition}
\begin{proof}
 If $f$ is continuous and bounded on $I$, then $\langle \omega, f\rangle=\int_a^b \omega(x) f(x) \, \mathrm{d}x$ is a limit of Riemann sums, $\sum_i \omega(x_i)f(x_i) \delta x$, each of which is sub-Gaussian as a finite sum of sub-Gaussian random variables. Further,
    \begin{align*}
        \left \| \sum_i f(x_i) \omega(x_i) \delta x\right \|_{\psi_2} \le \sum_i |f(x_i)| K \delta x \overset{\delta x \to 0}\longrightarrow K \|f \|_{L^1}
    \end{align*}
    and so the limit is also sub-Gaussian with sub-Gaussian norm at most $K \|f\|_{L^1}$. To arrive at the stated inequality, note that $L^2(I) \subset L^1(I)$ and
    \begin{align*}
        \|f\|_{L^1}
        &= \int_I f\,\mathrm{sign}(f)\\
        &\le \|\mathrm{sign}(f)\|_{L^2}\|f\|_{L^2}\\
        &\le \sqrt{b-a}\|f\|_{L^2}.
    \end{align*}

    If $f$ is not continuous, take a continuous sequence $f_n \overset{L^2}{\to} f$ with $\|f_n\|_{L^2} \le \|f\|_{L^2}$. Then $\langle \omega, f_n\rangle \to \langle \omega,f\rangle_{L^2}$, and each $\langle \omega,f_n\rangle_{L^2}$ has sub-Gaussian norm at most $K \sqrt{b-a}\|f\|_{L^2}$. Therefore, $\|\langle \omega, f\rangle_{L^2}\|_{\psi_2} \le K \sqrt{b-a}\|f\|_{L^2}$ also.
\end{proof}

In particular, this means that $\omega$ is a sub-Gaussian random vector in the sense of Definition~\ref{def:subg-vec-sketches} with sub-Gaussian norm at most $K \sqrt{b-a}$.

Now, is it the case that sub-Gaussian random vectors in the sense of Definition~\ref{def:subg-vec-sketches} are sub-Gaussian pointwise? Not necessarily; let $X \sim \mathrm{Exp}(1)$, and define $\omega(x)= X\mathbf 1_{X \le 1/x}$ for $x>0$ and $\omega(0)=X$. Then $\omega(x)$ is an element of $L^2([0,\infty))$ and
$$
\| \omega\|_{L^2}^2 = \int_0^\infty X^2 \mathbf 1_{X \le 1/x} \mathrm d x = X^2\int_0^{1/X}\mathrm dx = X.
$$
Therefore for all $\|f\|_{L^2}=1$, 
$$
|\langle \omega,f\rangle |\le \| \omega \|_{L^2}\|f\|_{L^2} = X^{1/2}
$$
by the Cauchy--Schwarz inequality. Since $X$ is exponential and in particular sub-exponential, $X^{1/2}$ is sub-Gaussian~\cite{vershynin2012introduction}, and hence so is $\langle \omega, f\rangle$. Since this holds for all unit $f$, $\omega$ has finite sub-Gaussian norm in the sense of Definition~\ref{def:subg-vec-sketches}. However, $\omega(0)=X$ is exponentially distributed, which is not sub-Gaussian. In fact, while $\omega(x)$ is bounded for every $x>0$, we have $\| \omega(x)\|_{\psi_2} = O(x^{-1/2})$ as $x \to 0$.

Sketch functionals are also naturally linked to stochastic processes via stochastic calculus. While this is not the primary focus of this paper, we would like to note the following.

\begin{remark}
    Let $B(t)$ be standard Brownian motion. Define $\ess(f) = \int_0^\infty f(t)dB(t)$ according to the Wiener integral~\cite{nualart2006malliavin}. Then $\ess$ is an isotropic Gaussian sketch: for all $f$, $\ess(f)$ is a centred Gaussian random variable with variance $\|f\|^2_{L^2}$.
\end{remark} 

Again, what happens in the non-Gaussian case? Once more, Brownian motion (the Gaussian case) has some remarkably useful properties that simply do not hold more generally, but this is still an interesting question.

\begin{definition}{(Stochastic Riemann Integral).}\label{def:stochastic_integral}
    Let $f:I \to \mathbb R$ and $X(t),t \in  I$ be a stochastic process. Whenever the limit exists taking the mesh size $\delta t \to 0$ of a mesh for $I$, define
    $$
    \int_I  f(t) dX(t) = \lim_{\delta t \to 0}\sum_{i} f(t_i)(X_{t_i+1}-X_{t_i}),
    $$
    where the convergence is in $L^2(\mathbb P)$.
\end{definition}

\begin{proposition}\label{prop:skfun_integration}
    Let $0 < a < \infty$ and $f\in L^2([0,a] )$ be Riemann integrable. Let $\ess$ be a centred sketch functional on $L^2([0,a])$ with bounded covariance operator $C$. Define the stochastic process $X(t)=\ess( \mathbf1_{[0,t)})$ for $t \le a$. Then $X$ is mean-square continuous and the limit in Definition~\ref{def:stochastic_integral} exists, with convergence in $L^2(\mathbb P)$. Further,
    $$ 
        \int_0^a f(t)\mathrm dX(t) = \ess(f).
    $$
    \begin{proof}
The fact that $X$ is mean-square continuous follows from its covariance operator being bounded:
\begin{equation*}
    \mathbb E[|X(t+u)-X(t)|^2]=\mathbb E[\ess(\mathbf1_{[t,t+u)})^2]= \|C^{1/2}\mathbf 1_{[t,t+u)}\|^2_{L^2} \le u\|C\|_2\text{.}
\end{equation*}
So, $\mathbb E[|X(t+u)-X(t)|^2] \to 0$ as $u \to 0$. 

For the latter part of the proposition, we have that

\begin{align*}
    \lim_{\delta t \to 0} \sum_{i} f(t_i)(X_{t_i+1}-X_{t_i}) &= \lim_{\delta t \to 0} \sum_{i} f(t_i)\ess(\mathbf1_{[t_i,t_i+1)}) \\
    &= \lim_{\delta t \to 0} \sum_{i} \ess(f(t_i) \mathbf 1_{[t_i,t_{i+1})})\\
    &= \lim_{\delta t \to 0}  \ess\left(\sum_{i} f(t_i)\mathbf 1_{[t_i,t_{i+1})}\right)\\
    &= \ess\left(\lim_{\delta t \to 0} \left [\sum_{i} f(t_i)\mathbf 1_{[t_i,t_{i+1})}\right ]\right)
\end{align*}
and the inner limit converges to $f$ as an element of $L^2([0,a])$. The claim follows as $\ess$ is a continuous map from $L^2([0,a])$ to $L^2(\mathbb P)$.
    \end{proof}
\end{proposition}

This shows that in some way sketch functionals with bounded covariance operators can be realised as integrals with respect to a mean square continuous stochastic process. This is not a particularly strong result --- mean-square continuous processes aren't even guaranteed to have continuous sample paths --- but we can't hope for much stronger. For example, even well-behaved sketch functionals like the Rademacher sketch usually yield $X(t)$ that fail to be local martingales, and the rich theory of It\^{o} calculus simply does not apply. In the sub-Gaussian case however, sub-Gaussian sketch functionals do correspond to pointwise sub-Gaussian processes; $X_t=\ess(\mathbf 1_{[0,t)})$ is sub-Gaussian with norm at most $K\sqrt{t}$. Note that $X_t$ differs from point values of $\omega$, so there is no contradiction with the discussion above. In Figure~\ref{fig:sample_paths}, we compare the behaviour of some of these processes for different sketch distributions. 

\begin{figure}[htbp]
    \centering
    \includegraphics[width=0.5\linewidth]{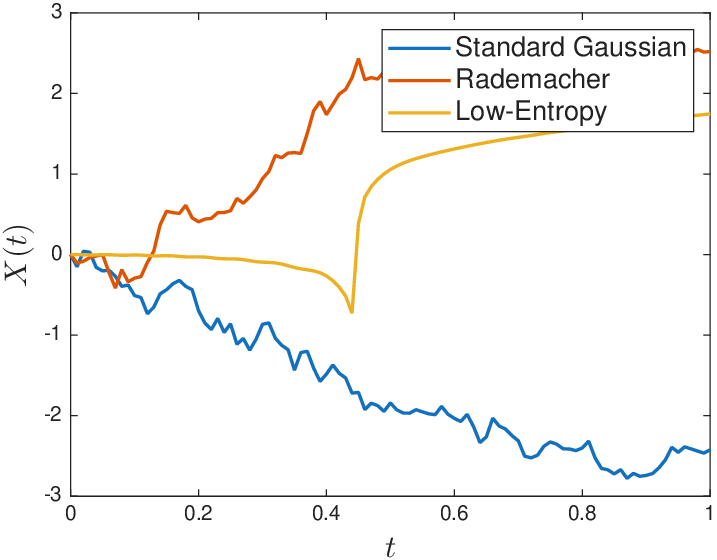}
    \caption{Sample paths of $X(t)$ defined in Proposition~\ref{prop:skfun_integration} for a variety of sketch functionals. When standard Gaussian random variables are used, the result is genuine Brownian motion. Instead using iid Rademacher coefficients, the result is a pointwise sub-Gaussian process with $\|X(t)\|_{\psi_2} \le \sqrt t$ that resembles Brownian motion but isn't. The ``low-entropy sketch'' is an example of an isotropic distribution that depends only on a one-dimensional random variable, defined in Section~\ref{subsection:discrete_schrod}.}
    \label{fig:sample_paths}
\end{figure}

\subsection{Low-Rank Operator Approximation Results}\label{subsection:mainresults}

At this point, we have all the necessary definitions to state the main results of this paper, which are tail bounds on the error of the infinite-dimensional randomised SVD used with sub-Gaussian sketches. The proofs refer to some results to be proven in the next subsection.

Similar to the Gaussian case, let's start with isotropic sketches and extend.

\begin{theorem} \label{theorem:subgaussian_maintheorem}
    Let $A: \mathcal H \to \mathcal K$ be the Hilbert--Schmidt target operator between separable Hilbert spaces $\mathcal H, \mathcal K$ and having singular value decomposition $U \Sigma V^T$. Sample $\ell$ independent isotropic sketch functionals with sub-Gaussian norm at most $K$ (Definition~\ref{def:subg-vec-sketches}) and let $\hat A$ be the output of Algorithm~\ref{alg:boullealg} applied with this sketch. 
    
    Let $u,t \ge 0$ be user-chosen parameters. Take $a,b$ given in Remark~\ref{rem:tailbound_constants}, and let $k<\ell$ be such that $D=\sqrt \ell - K \sqrt k - a-bt > 0$. Adopt the notation of Lemma~\ref{lemma:inf_dim_structural_bound} for $\Sigma_\perp$, $\Ess_1$ and $\Ess_2$. Then, with probability at least $1-\exp\big((\ell/6)(-u+1/2)\big)-\exp(-t^2/2)$,

\begin{equation}
    \|A - \hat A\|_{HS}^2 \le \|\Sigma_\perp\|_{HS}^2\left(1+\frac{\ell(1+K^2u)}{(\sqrt \ell - K\sqrt k - a-bt)^2} \right).
\end{equation}

\begin{proof}
    It is sufficient to bound $\| \Sigma_\perp \Ess_2 \Ess_1^\dagger\|_{HS}^2 \le \| \Sigma_\perp \Ess_2\|_{HS}^2 \| \Ess_1^{\dagger}\|^2_{2}$. We bound the factors separately.

    \underline{Step 1: Bounding $\| \Sigma_\perp \Ess_2\|_{HS}$.}
    
We apply Theorem~\ref{theorem:hs_bound} to $\Sigma_\perp V_\perp^T \Ess$. In this case, the covariance operator is $C= \Sigma_\perp V_\perp^TV_\perp \Sigma_\perp$ so that $\tr(C)=\|\Sigma_\perp\|_{HS}^2$, and
\begin{align*}
    \Sigma_\perp V_\perp^T \Ess &= \Sigma_\perp \big( \ess(v_{k+1}), \ess(v_{k+2}), \dots \big)\\
    &= \sum \sigma_{k+i} \ess(v_{k+i})e_i
\end{align*}
where $e_i$ is the natural basis of $\ell^2(\mathbb R)$. This matches the required expansion form as $\|\ess(v_j)\|_{\psi_2}\le K$ for all $j$ by assumption.

So, with probability at least $1-\exp\big((\ell/6)(-u+1/2)\big)$, we have
\begin{equation*}
    \|\Sigma_\perp \Ess_2\|_{HS}^2 \le \ell \|\Sigma_\perp\|_{HS}^2(1+K^2u).
\end{equation*}
    
    \underline{Step 2: Bounding $\|\Ess_1^\dagger\|_2$.}
    
    The columns of $\Ess_1$, which we remember is $k \times \ell$, have covariance matrix $V_k ^T \mathbb I_{\mathcal H} V_k= \mathbb I_k$ and sub-Gaussian norm at most $K$. So, applying Theorem~\ref{theorem:twosided} to $\Ess_1^T$, we have with probability at least $1-\exp(-t^2/2)$ and $a,b$ given in Remark~\ref{rem:tailbound_constants},
    \begin{align*}
        \sigma_{\min}( \Ess_1) &\ge \sqrt{\ell
        } - K \sqrt k - a-bt\text{,}
    \end{align*}
    and hence
    \begin{equation*}
        \|\Ess_1^\dagger\|_2^2 \le \frac{1}{(\sqrt \ell - K \sqrt k - a - bt)^2}\text{.}
    \end{equation*}

So, by sub-multiplicativity and a union bound, we have that with probability at least $1-\exp\big((\ell/6)(-u+1/2)\big)-\exp(-t^2/2)$,

\begin{equation*}
    \| \Sigma_\perp \Ess_2 \Ess_1^\dagger \|_{HS}^2 \le \frac{\ell\|  \Sigma_\perp\|_{HS}^2(1+K^2u)}{(\sqrt \ell - K \sqrt k - a - bt)^2}.
\end{equation*}

Then bound $\|A - \hat A\|_{HS}^2 \le \| \Sigma_\perp \|^2+\|\Sigma_\perp \Ess_2 \Ess_1^\dagger\|^2$ by Lemma~\ref{lemma:inf_dim_structural_bound}.
\end{proof}
\end{theorem}

\begin{remark}\label{remark:horrific_k_dependence}
    The oversampling parameter $\ell/k$ in Theorem~\ref{theorem:subgaussian_maintheorem} is governed by the denominator $D$, inherited from Theorem~\ref{theorem:twosided}. Theorem~\ref{theorem:twosided} implies that $\ell/k$ must be greater than $K^2$ for large $k$ and $\ell$ to ensure that $\Ess_1$ is well-conditioned with high probability. This is a significant difference from the Gaussian case, where $\ell/k$ can be arbitrarily close to $1$. It was proved in \cite{webb2026sharp} that within this general matrix family the $K^2$ oversampling requirement is sharp.
\end{remark}

The most comparable result in the literature is from Saibaba--Mi{\k{e}}dlar~\cite{saibaba2025randomized}, which also treats subspace iteration and a wider variety of distributions including those with bounded second moments. Most of their results are dependent on the ambient dimension and hence would not be valid in our setting. However, they do give the following dimension-independent result for sub-Gaussian sketch vectors:

\begin{theorem}
    Let $A \in \mathbb R^{m \times n}$ with SVD $U \Sigma V^T$, and let $\varepsilon, \delta$ be user-chosen parameters. Let $k$ be the target rank and let
    $$
    \ell = \frac{C_{CS}K^4}{\varepsilon^2}\left(\sqrt{k} + \sqrt{\log\frac 4 \delta}\right)^2.
    $$ 
    Let $\Ess \in \mathbb R^{n \times \ell}$ be a random matrix whose columns are independent isotropic sub-Gaussian random vectors with sub-Gaussian norm at most $K$, where 
    
    Let $\hat A$ be the output of Algorithm~\ref{alg:hmt} applied to $A$ with the sketch $\Ess$. Then with failure probability at most $\delta$,

    $$
    \| A - \hat A\|_2^2 \le \|\Sigma_\perp\|_2^2\left(1 + \frac{\big(\sqrt \ell + C_{CB}K^2 ( \frac{\| \Sigma_\perp\|_{F}}{\| \Sigma_\perp\|_2} + \sqrt{\log \frac 4 \delta})\big)^2}{(1-\varepsilon)^2\ell}\right),
    $$
    where $C_{CS}, C_{CB}$ are both constants and $\| \cdot \|_{F}$ denotes the Frobenius norm.
\end{theorem}

This result gives an explicit dependence on the failure probability $\delta$ and treats the spectral norm where we treat the Hilbert--Schmidt norm, whose finite-dimensional analogue is the Frobenius norm. However, it suggests the required sampling is $\ell = O(K^4 k)$ where we have managed $\ell = O(K^2 k)$ and relies on unspecified constants.

Our more detailed analysis has resolved the issue of unspecified constants, a question that~\cite{saibaba2025randomized} gives in the abstract. The underlying argument is the same, though; to bound $\| \Sigma_\perp \Ess_2 \Ess_1^\dagger\| \le \|\Sigma_\perp \Ess_2 \| \| \Ess_1^\dagger\|$ and deal with the terms separately. This disregards cross-cancellation, meaning our results share the limitation that the floor of our error bounds are $\| A - \hat A\| \le \sqrt 2 \|\Sigma_\perp \|$ as $\ell \to \infty$.

Contrast with the comparable Gaussian result of Kressner et al.\cite{KRESSNER2026259}, where if we take the oversampling parameter $p \to \infty$ then their bounds on $\| \Sigma_\perp \Ess_2 \Ess_1^\dagger\|_{HS}$ tend to $0$ rather than $\|\Sigma_\perp\|_{HS}$ as ours do. In the Gaussian case, it is possible to deal with $\Ess_2$ and $\Ess_1$ simultaneously as they are independent. This approach accounts for cross-cancellation between the two terms and predicts the correct behaviour in the large-$\ell$ limit. Existing bounds on the smallest singular values of sub-Gaussian random matrices are also worse than corresponding Gaussian results; where they are allowed $\ell=k+p$, the best existing guarantees for sub-Gaussian random matrices in the best-case $K=1$ require $\ell-k = O(k^{1/2})$ to leading order for $\Ess_1$ to be well-conditioned.

\begin{theorem} \label{theorem:subgaussian_cov}
 Let $A:\mathcal H \to \mathcal K$ be a Hilbert--Schmidt target operator with SVD $U \Sigma V^T$. Sample $\ell$ independent sketch functionals of the form $C^{1/2} \omega$ where $C$ is a bounded SPSD operator on $\mathcal H$ and $\omega$ are isotropic with sub-Gaussian norm at most $K$, and let $\hat A$ be the output of Algorithm~\ref{alg:boullealg} applied with this sketch. 
 
 Let $u,t \ge 0$ be user-chosen parameters. Take $a,b$ from Remark~\ref{rem:tailbound_constants} and let $k < \ell$ be such that $D>0$, as in Theorem~\ref{theorem:subgaussian_maintheorem}. Adopting the notation of Lemma~\ref{lemma:inf_dim_structural_bound}, assume $V_k^TC^{1/2}$ has rank $k$.
If $\|\Sigma_\perp\|_{HS}>0$, define
$$
\beta_k = \frac{\|\Sigma_\perp V_\perp^T C^{1/2}\|_{HS}}{\sigma_k(V_k^TC^{1/2})\|\Sigma_\perp\|_{HS}}\text{.}
$$
    
    Then, with probability at least $1-\exp\big((\ell/6)(-u+1/2)\big)-\exp(-t^2/2)$,
\begin{equation}
    \|A - \hat A\|_{HS}^2 \le \|\Sigma_\perp\|_{HS}^2\left(1 + \beta_k^2~ \frac{\ell(1+K^2u)}{(\sqrt \ell - K \sqrt k - a - bt)^2}\right)\text{.}
\end{equation}
If $\|\Sigma_\perp\|_{HS}=0$, then $\hat A=A$ with at least the same probability, without defining $\beta_k$.

\begin{proof}
    When $\|\Sigma_\perp\|_{HS}=0$, Lemma~\ref{lemma:inf_dim_structural_bound} gives $\hat A=A$ on the full-row-rank event for $\Ess_1$ established below. Thus it suffices to consider $\|\Sigma_\perp\|_{HS}>0$ for the bound involving $\beta_k$.

    Again, consider the structural bound $\| \Sigma_\perp \Ess_2 \Ess_1^\dagger\|_{HS}^2 \le \|\Sigma_\perp \Ess_2\|_{HS}^2 \| \Ess_1^\dagger\|_2^2$.
    
    \underline{$\Sigma_\perp \Ess_2$:}
    We have that $\Sigma_\perp \Ess_2 = \Sigma_\perp V_\perp^T C^{1/2} \Omega$, where $\Omega$ has isotropic sub-Gaussian columns $\omega$ with sub-Gaussian column norm $K$. This has covariance operator $\Sigma_\perp V_\perp^T C V_\perp \Sigma_\perp$. 
    
    Let
    \begin{equation*}
        B=\Sigma_\perp V_\perp^T C^{1/2}
        =\sum_i \lambda_i^{1/2}f_i g_i^T
    \end{equation*}
    be its singular value decomposition, noting that $B$ is Hilbert--Schmidt because $\Sigma_\perp$ is. Representing the columns $\omega$ in the $g_i$ basis,
    \begin{align*}
        \omega(x)&=\sum_i \omega(g_i)\langle g_i,x\rangle\\
        &=\sum_i \ess_i\langle g_i,x\rangle,
    \end{align*}
    where $\ess_i=\omega(g_i)$ are uncorrelated with unit variance and $\|\ess_i\|_{\psi_2}\le K$. Thus,
    \begin{align*}
    B\omega(x)&=\omega(B^Tx)\\
    &=\sum \ess_i \langle g_i,B^Tx \rangle \\
    &= \sum \ess_i \langle Bg_i,x\rangle \\
    &= \sum \lambda_i^{1/2} \ess_i \langle f_i, x\rangle.
    \end{align*}
Applying Theorem~\ref{theorem:hs_bound} to $\Sigma_\perp \Ess_2$, we have that
\begin{equation*}
    \mathbb P\big(\| \Sigma_\perp \Ess_2\|_{HS}^2 > \ell ~\tr(\Sigma_\perp V_\perp^T C V_\perp \Sigma_\perp)(1+K^2u) \big) \le \exp\bigg (\frac{\ell}{6}\big(\frac12-u\big) \bigg) \text{.}
\end{equation*}

\underline{$\Ess_1^\dagger$:}
We have $\Ess_1=V_k^TC^{1/2}\Omega$ under the same assumptions as above on $\Omega$. Taking the columns of $W_k$ to be an orthonormal basis for the support of $V_k^TC^{1/2}$ gives
\begin{equation*}
    V_k^TC^{1/2}=V_k^TC^{1/2}W_kW_k^T.
\end{equation*}
Consequently,
\begin{align*}
    \|\Ess_1^\dagger\|_2
    &=\sigma_k(V_k^TC^{1/2}W_kW_k^T\Omega)^{-1}\\
    &\le\sigma_k(V_k^TC^{1/2}W_k)^{-1}\sigma_k(W_k^T\Omega)^{-1}.
\end{align*}

$\sigma_k(V_k^TC^{1/2}W_k)=\sigma_k(V_k^TC^{1/2})$ and $W_k^T\Omega$ is a $k\times \ell$ matrix with isotropic sub-Gaussian columns whose norm does not exceed $K$, so is subject to Theorem \ref{theorem:twosided}. Thus with probability at least $1-\exp(-t^2/2)$,

\begin{equation*}
    \| \Ess_1^\dagger \|_2^2 \le \sigma_k(V_k^TC^{1/2})^{-2}(\sqrt \ell - K \sqrt k - a - bt)^{-2}\text{.}
\end{equation*}

By a union bound, with probability at least $1-\exp((\ell/6)(-u+1/2))-\exp(-t^2/2)$,
\begin{align*}
    \big \| \Sigma_\perp \Ess_2 \Ess_1^\dagger \big \|^2_{HS} &\le \frac{\ell~\tr (\Sigma_\perp V_\perp^TCV_\perp\Sigma_\perp)(1+K^2u)}{ \sigma_k(V_k^TC^{1/2})^2(\sqrt \ell - K \sqrt k - a - bt)^2}\\
    &=  \frac{\tr(\Sigma_\perp V_\perp^T C V_\perp \Sigma_\perp)} {\lambda_k(V_k^TCV_k)} ~\frac{\ell(1+K^2u)}{(\sqrt \ell - K \sqrt k - a - bt)^2}.
\end{align*} 

Once again, finish by taking the structural bound $\|A - \hat A\|_{HS}^2 \le \| \Sigma_\perp \|^2+\|\Sigma_\perp \Ess_2 \Ess_1^\dagger\|^2$ and factorise.
\end{proof}
\end{theorem}

Similar to results in~\cite{boulle2021generalization}, the bounds rely on $\beta_k$ which measures how well the covariance operator aligns with the target operator, although our ``covariance quality measure'' is different to theirs.  We note that we haven't chosen to frame this result in terms of the sub-Gaussian norm of the sketch itself $\| \ess \|_{\psi_2}$, but the sub-Gaussian norms of the components $\ess_i$. The sub-Gaussian norm and second moment are intertwined but not equivalent, so we have found that encoding covariance information and sub-Gaussian norm information separately gives the cleanest result as this allows us to apply the reasoning of~\ref{theorem:gaussian_cov} more directly.

We also give results for the infinite-dimensional Nystr\"{o}m method, another randomised low-rank approximation algorithm for SPSD operators.

\begin{algorithm} 
\caption{Infinite-Dimensional Nystr\"om Algorithm (General Covariance)}
\label{alg:nystrom}
\textbf{Input:} covariance operator $C$, SPSD  trace-class target operator $A \in \mathrm{TC}(\mathcal H)$, sketch size $\ell$\\
\textbf{Output:} rank-$\ell$ approximation $\hat A$
\begin{algorithmic}[1]
\State Sample independent sketch functionals $\Ess = (\ess_1, \ess_2, \dots, \ess_\ell)$
\State $Y \gets A \Ess$ (according to Definition~\ref{definition:generalised_sketch_matrices})
\State $\hat A \gets Y ( \Ess^TY)^\dagger Y^T$
\end{algorithmic}
\end{algorithm}

\begin{corollary}\label{corollary:nystrom}
    Let $A:\mathcal H \to \mathcal H$ be a SPSD trace-class target operator with orthonormal eigendecomposition $V \Sigma V^T$. Sample $\ell$ independent isotropic sub-Gaussian sketch functionals with sub-Gaussian norm at most $K$ and let $\hat A$ be the output of Algorithm~\ref{alg:nystrom} with this sketch. 
    
    Write $\|A\|_* = \tr((A^TA)^{1/2})$ for the nuclear norm and let $u,t \ge 0$ be user-chosen parameters. Take $a,b$ from Remark~\ref{rem:tailbound_constants} and let $k < \ell$ be such that $D>0$, as in Theorem~\ref{theorem:subgaussian_maintheorem}. Adopt the notation of Lemma~\ref{lemma:inf_dim_structural_bound} for $\Sigma_\perp$, $V_k$, $V_\perp$.

    Then with probability at least $1-\exp\big((\ell/6)(-u+1/2)\big)-\exp(-t^2/2)$, 
     \begin{equation}
    \|A - \hat A\|_* \le \|\Sigma_\perp\|_*\left(1 + \frac{\ell(1+K^2u)}{(\sqrt \ell - K \sqrt k - a - bt)^2}\right).
    \end{equation}
    For independent sketch functionals $\ess$ of the form $C^{1/2}\omega$, where $C$ is a bounded SPSD operator on $\mathcal H$ and $\omega$ is isotropic with sub-Gaussian column norm $K$, assume $V_k^TC^{1/2}$ has rank $k$. If $\|\Sigma_\perp\|_*>0$, then with at least the same probability,

\begin{equation}
\|A - \hat A\|_{*} \le \|\Sigma_\perp\|_{*}\left(1 + \hat \beta_k \frac{\ell(1+K^2u)}{(\sqrt \ell - K \sqrt k - a - bt)^2} \right),
\end{equation}
where $\hat \beta_k = \beta_k(A^{1/2})^2$.
If $\|\Sigma_\perp\|_*=0$, then $\hat A=A$ with at least the same probability in either sketch setting, without defining $\hat\beta_k$.

\begin{proof}
    For any SPSD operator $A$ with sketch $\Ess$, denote $Z = A^{1/2} \Ess$. Then the Nystr\"om approximation for $A$ can be rewritten as
    \begin{equation*}
        \hat A = A \Ess (\Ess^T A \Ess)^{\dagger}(A\Ess)^T = A^{1/2} Z (Z^T Z)^{\dagger}Z^T A^{1/2} \\
    \end{equation*}
    and $P_Z=Z (Z^TZ)^{\dagger}Z^T$ is a symmetric projection onto the columns of $Z$. Now observe that $P_ZA^{1/2}$ is a randomised SVD for $A^{1/2}$ with  sketch $\Ess$, and since $A^{1/2}$ can also be taken to be SPSD,
    \begin{align*}
        \|A^{1/2}-P_ZA^{1/2}\|^2_{HS}&=\tr((A^{1/2}-P_ZA^{1/2})^T(A^{1/2}-P_ZA^{1/2})) \\
        &= \tr (A-A^{1/2}P_ZA^{1/2} -A^{1/2}P_ZA^{1/2}+A^{1/2}P_ZP_ZA^{1/2})\\
        &= \tr(A - A^{1/2}P_ZA^{1/2}) = \|A-\hat A\|_*,
    \end{align*}
    because $A- A^{1/2}P_ZA^{1/2}$ is positive semidefinite. The result follows by applying either Theorem~\ref{theorem:subgaussian_maintheorem} or Theorem~\ref{theorem:subgaussian_cov}, noting that $\|A^{1/2}\|_{HS}^2=\|A\|_*$. When $\|\Sigma_\perp\|_*=0$, the full-row-rank event for $V_k^T\Ess$ and Lemma~\ref{lemma:inf_dim_structural_bound} give $P_ZA^{1/2}=A^{1/2}$, so $\hat A=A$.
\end{proof}
\end{corollary}
 
\begin{remark}
    Algorithm~\ref{alg:nystrom} uses $\Ess^T A \Ess$ in its definition, which requires some care to define. Concretely, since $A^{1/2}$ is Hilbert-Schmidt, $Z=A^{1/2} \Ess$ has $\mathcal H$-valued columns almost surely, and we may define $\Ess^T A \Ess = Z^* Z$. In practice, we calculate $\Ess^T ( A \Ess)$, which does not require access to $A^{1/2}$ and coincides with that definition algebraically.
\end{remark}

\subsection{Sub-Gaussian Random Matrix Theory}

In view of Lemma~\ref{lemma:inf_dim_structural_bound}, we aim to bound $\|\Sigma_\perp \Ess_2 \Ess_1^\dagger\|$. In the Gaussian case, the usual approach is to condition on $\Ess_1$ and bound $\|\Sigma_\perp \Ess_2 B\|$ for deterministic $B$. This conditioning is convenient because uncorrelated Gaussian variables are independent. In the sub-Gaussian case, we do not have this luxury, and instead take $\|\Sigma_\perp \Ess_2 \Ess_1^\dagger\| \le \| \Sigma_\perp \Ess_2\| \|\Ess_1^\dagger\|$ and bound the quantities separately.\\
So, we will need a bound on the smallest singular value of sub-Gaussian random matrices for a bound on the pseudoinverse, and we will also need a bound for products with a sub-Gaussian random matrix. 

The following estimate on the mgf of $X^2-1$ for a sub-Gaussian random variable $X$ will be of use to us.
\begin{lemma}\label{lem:sub-exp bound}
Let $X$ be a centred, isotropic sub-Gaussian random variable with sub-Gaussian norm $K$. Then there exists a constant $C > 0$ such that
\[
\mathbb{E}\left[\exp\left(t\left(X^2 - 1\right)\right)\right]
\le
\exp\left(3t^2K^4\right) \quad \text{ if } \quad 0 \leq t \leq \frac{1}{6K^2}.
\]
\begin{proof}
  Let $G\sim N(0,1)$ be independent of $X$. Since $\mathbb{E}[e^{tG}] = e^{t^2 /2}$ and $\mathbb{E}[e^{tG^2}] = (1-2t)^{-1/2}$, the standard Gaussian averaging argument gives us
\begin{equation}\label{eq:gaussian-mgf}
 \mathbb{E} [e^{sX^2}]
 =\mathbb{E} [e^{\sqrt{2s}\,GX}]
 \leq\mathbb{E} [e^{K^2sG^2}]
 =(1-2K^2s)^{-1/2},
 \qquad 0\leq s<\frac1{2K^2}.
\end{equation}
Set $R(u)=e^u-1-u$. We first bound
$\mathbb{E} [R(X^2/(6K^2))]$ using the elementary inequality
\begin{equation}\label{eq:exponential-remainder}
 R(u)\leq\frac{17}{150}(e^{2u}-1),\qquad u\geq0,
\end{equation}
which is easily verified numerically. Combining~\eqref{eq:exponential-remainder} with~\eqref{eq:gaussian-mgf} gives
\[
 \mathbb{E} \left[R\!\left(\frac{X^2}{6K^2}\right)\right]
 \leq\frac{17}{150}(\sqrt3-1)\leq\frac1{12}.
\]
Since $\log x\leq x-1$ and $\mathbb{E} X^2=1$,
\begin{equation*}
    \log\mathbb{E}[e^{t(X^2-1)}] =\log\mathbb{E} [e^{tX^2}]-t \leq\mathbb{E}[e^{tX^2}]-1-t = \mathbb{E} [e^{tX^2}-1-tX^2] = \mathbb{E} [R(tX^2)].
\end{equation*}
The exponential series gives $R(\theta u)\leq\theta^2R(u)$
for $u\geq0$ and $0\leq\theta\leq1$. For $0\leq t\leq1/(6K^2)$,
we conclude
\[
 \log\mathbb{E} [e^{t(X^2-1)}]
 \leq\mathbb{E} [R(tX^2)] \leq(6K^2t)^2\mathbb{E} R\!\left(\frac{X^2}{6K^2}\right)
 \leq3K^4t^2.
\]
\end{proof}
\end{lemma}

  Recall that if $\Ess$ has covariance operator $C$, then $T\Ess$ has covariance operator $T C T^T$, and so the following yields an estimate for $\|T \Ess\|_{HS}$. 

\begin{theorem}\label{theorem:hs_bound}
    Suppose $\ess = \sum_i \lambda_i^{1/2} \ess_i p_i$ where $C= \sum_i\lambda_i p_i p_i^T$ is an orthonormal eigendecomposition of the trace-class covariance operator and $\ess_i$ are centred sub-Gaussian random variables with unit variance and sub-Gaussian norm $K$. If $\Ess$ is formed from $\ell$ iid copies of this sketch, then
    \begin{equation}
        \mathbb E\left [ \| \Ess\|_{HS}^2 \right]=\ell~ \tr(C)
    \end{equation}
    and
        \begin{equation}
        \mathbb P\bigg ( \| \Ess\|^2_{HS} > \mathbb E[\|\Ess\|_{HS}^2]\big (1+K^2u\big )\bigg ) \le \exp\left(\frac16\ell\left(\frac12-u\right)\right)\text{.}
    \end{equation}
\end{theorem}
\begin{proof}
    If $\tr(C)=0$, Proposition~\ref{prop:traceclass} gives $\mathbb E[\|\ess\|^2]=0$, so $\Ess=0$ almost surely. Both strict upper-tail probabilities are therefore zero and the claims hold immediately. Henceforth assume $\tr(C)>0$ and omit indices with $\lambda_i=0$ from the following sums and products, since they contribute nothing to $\|\ess\|^2$.

    First, we consider one column $\ess$. We have that $\mathbb E[\|\ess\|^2]=\tr(C)$ by the proof of Proposition~\ref{prop:whitenoiseproblem} and for any $u\ge0$ and $\mu>0$,

    \begin{align*}
        \mathbb P \big ( \| \ess\|^2 > \mathrm{tr}(C)+K^2 \tr(C) u\big )&= \mathbb P\bigg ( \mu \sum_i \lambda_i \ess_i^2 > \mu\left (\sum_i \lambda_i + K^2 \tr(C) u  \right)\bigg )\\
        &= \mathbb P \bigg ( \exp\bigg( \mu \sum_i \lambda_i(\ess_i^2-1) \bigg)> \exp \big( \mu K^2 \tr(C) u\big) \bigg) \\
        &\le \exp(-\mu K^2 \tr(C) u) \mathbb E\left [ \exp \left( \mu \sum_i \lambda_i(\ess_i^2-1)\right) \right ] \\
        &\le \exp(- \mu K^2 \tr(C) u )\prod_i \mathbb E \left[\exp\left(\mu p_i\lambda_i(\ess_i^2-1)\right)\right]^{1/p_i}
    \end{align*}
    where the first inequality is Markov's inequality and the second is the generalised H\"older inequality with constants $p_i>0$ satisfying $\sum\frac1{p_i}=1$ to be chosen. Indeed, let $p_i = \frac{\tr(C)}{\lambda_i}$. Then,
    \begin{align*}
        \mathbb E\big[ \exp\big(\mu p_i \lambda_i (\ess_i^2-1)\big)\big]^{1/p_i}&= \mathbb E \big [ \exp\big(\mu \tr(C) (\ess_i^2-1)\big )\big ]^{\lambda_i/\tr(C)}\\
        &\le \exp( 3\mu^2\tr(C)^2 K^4)^{\lambda_i/\tr(C)}\\
        &= \exp(3 \mu^2 \tr(C) \lambda_i K^4)
    \end{align*}
by Lemma~\ref{lem:sub-exp bound} whenever $\mu \tr(C) \le \frac{1}{6K^2}$. So,
\begin{align*}
    \mathbb P \big ( \| \ess\|^2 > \mathrm{tr}(C)+K^2 \tr(C) u\big ) &\le \exp(- \mu K^2 \tr(C) u ) \prod_i \exp(3 \mu^2 \tr(C) \lambda_iK^4)\\
    &=  \exp\big(3\mu^2 \tr(C)^2K^4-\mu K^2 \tr(C)u\big)
\end{align*}
For moderate and large $u$, choosing $\mu = \frac1{6K^2 \tr(C)}$ minimises this expression within the permitted interval, which yields $\exp\big(\frac1{12}-\frac{u}{6}\big)$.

To extend to multiple independent columns $\Ess=(\ess^{(1)},\dots,\ess^{(\ell)})$, note that
\begin{align*}
    \mathbb E[\|\Ess\|_{HS}^2]
    =\sum_{j=1}^\ell\mathbb E[\|\ess^{(j)}\|^2] =\ell\,\tr(C).
\end{align*}
Applying Markov's inequality to the recentred sum $\sum_{j=1}^\ell(\|\ess^{(j)}\|^2-\tr(C))$ exactly as above gives
\begin{align*}
\mathbb P\big(\| \Ess\|_{HS}^2 > \mathbb E[\| \Ess\|_{HS}^2](1+K^2u) \big) &= \mathbb P\left ( \sum_{j=1}^{\ell} \big(\|\ess^{(j)}\|^2 - \tr(C) \big) > \ell \tr(C)K^2u\right )\\
&\le \exp(-\ell \mu K^2 \tr(C) u) \, \mathbb E\Big[ \exp\Big(\mu \sum_{j=1}^\ell (\| \ess^{(j)}\|^2-\tr(C))\Big)\Big] \\
&= \prod_{j=1}^\ell \exp(- \mu K^2 \tr(C)u) \, \mathbb E\big[\exp(\mu (\|\ess^{(j)
}\|^2 -\tr(C)))\big],
\end{align*}
so we bound each of these for $j=1,\ldots,\ell$ as above.
\end{proof}


As mentioned above, we will also need the following two-sided bound for singular values.

\begin{theorem}[\cite{webb2026sharp}]
\label{theorem:twosided}
Let $\Ess\in\mathbb R^{m\times n}$, $m\geq n\geq1$, have independent,
centred, isotropic rows with sub-Gaussian norm at most $K$ for some $K\geq1$ and put $L = \log ( e K)$. Then there exist constants $a$, $b$, $c$, $d$ depending on $K$ such that for all $t\ge 0$,
     \begin{align*}
         \mathbb P \left [ \sigma_{\min}(\Ess) \le \sqrt m - K \sqrt n  - a - b t\right] &\le e^{-t^2/2}\\
         \mathbb P \left [ \sigma_{\max}(\Ess) \ge \sqrt m+ K \sqrt n  + c + d  t\right] &\le e^{-t^2/2}.
     \end{align*}
\end{theorem}

\begin{remark}\label{rem:tailbound_constants}
The following choices of constants are valid.
\begin{center}
    \begin{tabular}{c|c|c|c|c}
         & $a$ & $b$ & $c$ & $d$  \\
         \hline
         $K=1$& $1.32$ & $1.12$ & $2.32$ & $\sqrt 2$ \\
         $K>1$ & $3.84 KL$ & $1.58K\sqrt{L}$ & $2.35KL$ & $1.5K\sqrt{L}$ \\
         
    \end{tabular}
\end{center}
\end{remark}

\section{Numerical Experiments}\label{section:numerical_exp}
In this section, we apply our theory to the randomised SVD of the continuous and discrete Schr\"odinger operators. The Schr\"odinger differential operator $E$ is defined for us as
\begin{equation}
    E\psi = - \epsilon\Delta \psi  + V \psi
\end{equation}
where $\psi$ belongs to the operator domain $D(E)$, a dense subspace of $L^2(I)$ (or $\ell^2(\mathbb Z)$ in Section~\ref{subsection:discrete_schrod}), with the specific domains given below. Here $I \subseteq \mathbb R^n$, $\Delta$ is the Laplacian, $V \ge0$ is a potential function and $\epsilon$ is some small positive constant.
Subject to certain boundary conditions, it is a symmetric positive semi-definite operator which is unbounded over $L^2$ with rich spectral theory, discussed in detail for example in Agmon~\cite{agmon1975spectral}. 
Its eigenvalues and eigenfunctions have significant interpretations in quantum physics. The situation is complicated because the definition involves the Laplacian $\Delta$, an unbounded operator defined only on a dense subset of $L^2$. We will instead work with the inverse of $E$, denoted $G$, so that $Gf=\psi$ if $E\psi = f$. 

We will also consider the discrete $\ell^2$ case by replacing $\Delta$ with the discrete Laplacian. The result is another SPSD operator where the spectral theory is different; boundedness now depends on the potential $V$. With analysis of the eigenvectors of the discrete Laplacian, we also exhibit an interesting example of an isotropic sketch entirely determined by a $1$-dimensional random variable.

\paragraph{Experimental setup.}
Experiments were conducted in MATLAB R2025b with Chebfun~\cite{driscoll2014chebfun}. In Section~\ref{subsection:cont_exp}, we represent $G$ using its first 500 eigenpairs (which omits an error of less than $10^{-3}$ in the Hilbert--Schmidt norm), isotropic sketches were represented in the normalised $L^2$ Legendre basis again to 500 components, and decaying covariance sketches were represented in the basis $\{ 1/\sqrt{2} \} \cup \{\sin(n \pi x), \cos(n \pi x)\}_{n \ge 1}$ up to $n=500$. We compute randomised SVDs from rank $\ell_{\min} = 50$ up to rank $\ell_{\max}=200$. In Section~\ref{subsection:discrete_schrod}, we represent $G$ with its first $1500$ eigenpairs (which also omits an error of less than $10^{-3}$ in the Hilbert--Schmidt norm) and all sketches are represented in the canonical $\ell^2(\mathbb Z)$ basis, with indices from -2000 to +2000. We compute randomised SVDs from $\ell_{\min}=50$ up to rank $\ell_{\max}=1000$. Throughout this section, error plots are over 15 trials sampling $\ell_{\max}$ sketch functionals, plotting medians and min-max error bands over the first $\ell$ samples, unless stated otherwise. Theorems~\ref{theorem:KPU_maintheorem},~\ref{theorem:bt_maintheorem} and \ref{theorem:gaussian_cov} were applied with $u=t=3$ and $p=5$ throughout. Our sub-Gaussian bounds were applied with $u=1,t=2$, and $k$ chosen to be maximal under the constraint that $t=3$ would be admissible. At rank $\ell_{\min}=50$, this target rank is $4$ and the choice of parameters guarantees a failure probability less than $0.14$.

\subsection{The Continuous Schr\"odinger Operator}\label{subsection:cont_exp}

In this subsection, we set $\epsilon = 0.1$ and find low-rank approximations to the Green's function solving
\begin{equation}
    E\psi =- \frac{1}{10} \Delta \psi + x^2 \psi = f, \label{eqn:example_cts_schro}
\end{equation}
with $f\in L^2([-1,1])$ and $\psi\in D(E)$, where $D(E)=H^2(-1,1)\cap H_0^1(-1,1)$. This domain imposes the homogeneous Dirichlet boundary conditions $\psi(-1)=\psi(1)=0$. The eigenvalues of $E$ are positive and, by the classical Sturm--Liouville eigenvalue asymptotics~\cite[Theorem~5.25]{teschl2012ordinary}, satisfy $\lambda_n(E)\sim(\pi^2/40)n^2$ as $n\to\infty$. Thus its inverse $G$ is trace-class and positive definite.

\begin{figure}[htbp]
    \centering
    \includegraphics[width=.45 \linewidth]{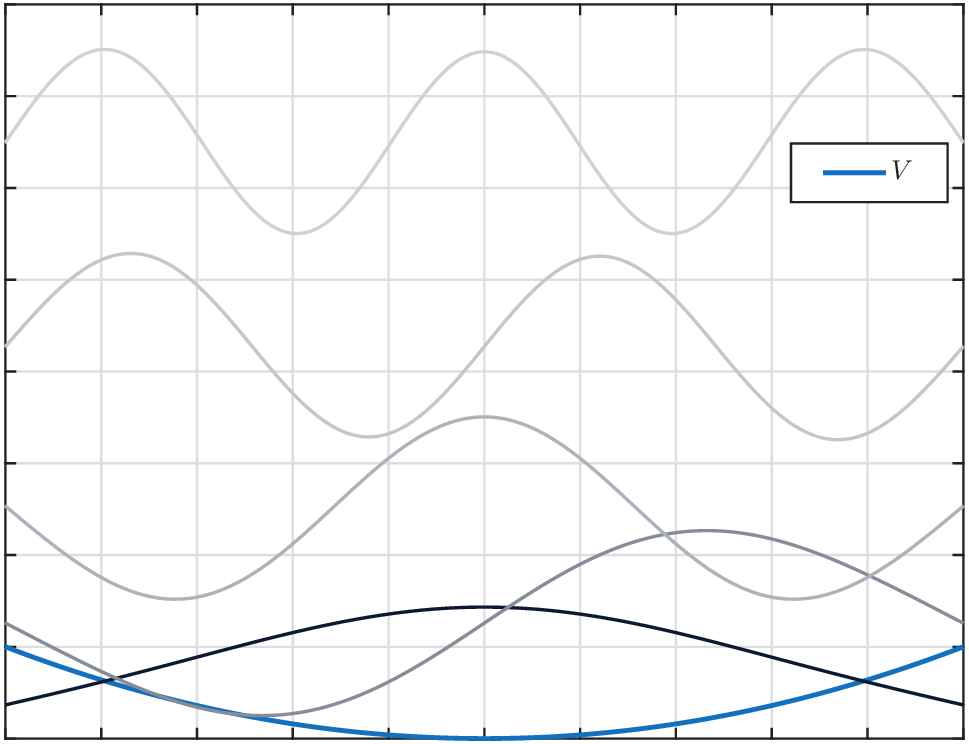}
    \caption{$\psi_n + \lambda_n$ for the first 5 eigenfunctions $\psi_n$ of continuous Schr\"odinger operator $E$ defined in~\eqref{eqn:example_cts_schro} together with the potential $V(x)$. Larger eigenvalues of $E$ (smaller eigenvalues of $G$) are given a lighter colour. Higher frequencies are associated with higher eigenvalues. }
    \label{fig:cts_schro_eigenfunctions}
\end{figure}

In Figure~\ref{fig:cts_schro_errors}, we apply the randomised SVD to $G$ with isotropic Rademacher sketch, and compare the resulting error to Theorem~\ref{theorem:subgaussian_maintheorem}. We compare this with an isotropic Gaussian sketch and Theorem~\ref{theorem:KPU_maintheorem}, using parameters given in the experimental setup paragraph. Note that in practice the distributions perform almost identically, although the sub-Gaussian results are much worse than their Gaussian counterparts, requiring much higher oversampling. 
 
\begin{figure}[htbp]
    \centering
    \includegraphics[width=0.4\linewidth]{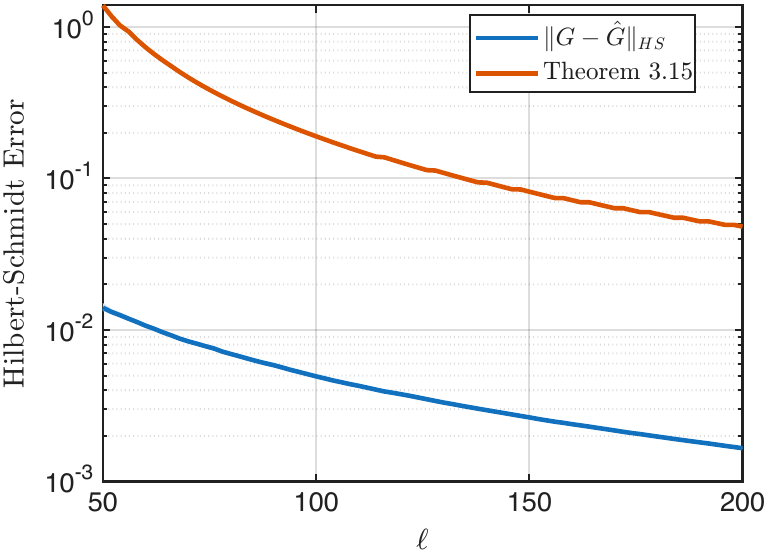}
    \includegraphics[width=0.4\linewidth]{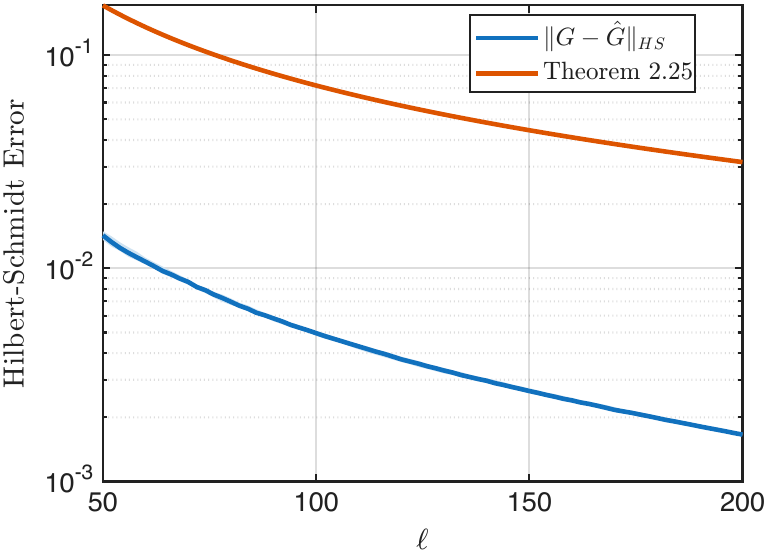}

    \caption{Practice vs.~theory errors when applying the randomised SVD to $G$ with isotropic Rademacher (left) and Gaussian (right) sketch functionals.}
    \label{fig:cts_schro_errors}
\end{figure}

In finite dimensions, most isotropic distributions perform very similarly, and this is true in the infinite-dimensional case too as shown in Figure~\ref{fig:distributions}. This contrasts with the huge influence of the sub-Gaussian norm $K$ suggested in Theorem~\ref{theorem:subgaussian_maintheorem} as discussed in Remark~\ref{remark:horrific_k_dependence}. That being said, it's clear that lowering the parameter $q$ in sparse Rademacher (increasing the sub-Gaussian norm) does have quite a large effect for low $\ell$.
\begin{figure}[htbp]
    \centering
    \includegraphics[width=0.5\linewidth]{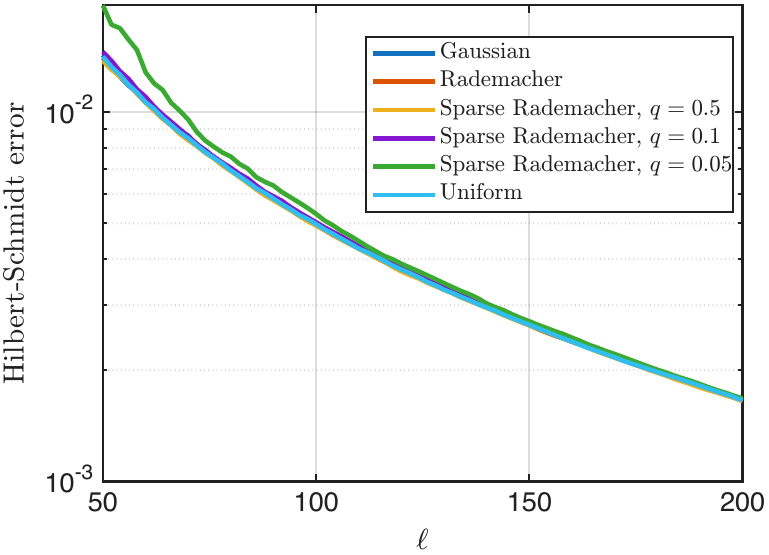}
    \caption{Hilbert--Schmidt errors when the randomised SVD is applied with a variety of isotropic distributions. Only the median over 15 trials is plotted here.}
    \label{fig:distributions}
\end{figure}

As noted in Figure~\ref{fig:cts_schro_eigenfunctions}, higher frequencies are associated with higher eigenvalues of $E$ and thus low eigenvalues of $G$. Perhaps it would be a good idea to encode this knowledge into our sketch. Let $C$ be an operator with eigenvectors $\sin(n\pi x)$ and $\cos(n\pi x)$ for $n\ge1$ each having eigenvalues $1/(n+1)^2$, and an eigenvalue of $1$ on constant functions. We apply the randomised SVD to $G$ using both a Rademacher and Gaussian sketch with this covariance operator, comparing to the theoretical bounds in Figure~\ref{fig:covariance_experiment}. The performance of the algorithm is noticeably better than an isotropic sketch, but no available bounds capture this improvement. We again plot Theorem~\ref{theorem:gaussian_cov} and Theorem~\ref{theorem:bt_maintheorem} with $p=5, u=t=3$. 

In Figure~\ref{fig:gaussian_cov_comparison}, we extend this idea and compare a variety of Gaussian sketches: isotropic, $\lambda_n = (n+1)^{-2}$ as before, and $\lambda_n = (n+1)^{-1}$ resulting in a compact but not trace-class covariance operator. It's clear that changing this decay can yield marginal gains in performance, although theoretically there is a tradeoff. If we've chosen the exact right-singular vectors $V$ and the covariance operator is $V \Lambda V^T$, then it is true that $\beta_k= \frac{\| \Sigma_\perp \Lambda^{1/2}_\perp\|_{HS}}{\lambda_k^{1/2}\| \Sigma_\perp\|_{HS}} \le \frac{ \lambda_{k+1}^{1/2}}{\lambda_k^{1/2}}$ which can decay arbitrarily quickly. Of course, it is a very lucky situation where we know the singular vectors in advance, and in this case there would be no need for a randomised SVD. Contrastingly, if the approximate singular basis is inexact and the decay is too strong, $\beta_k$ can quickly blow up. With Figure~\ref{fig:covariance_experiment} we are in such a case, which is why the sub-Gaussian Theorem~\ref{theorem:subgaussian_cov} lies below its Gaussian counterparts; the higher oversampling means a much smaller $\beta_k$. Optimising the Gaussian results over the oversampling parameter $p$ would reverse this picture.

\begin{figure}
    \centering
    \includegraphics[width=0.5\linewidth]{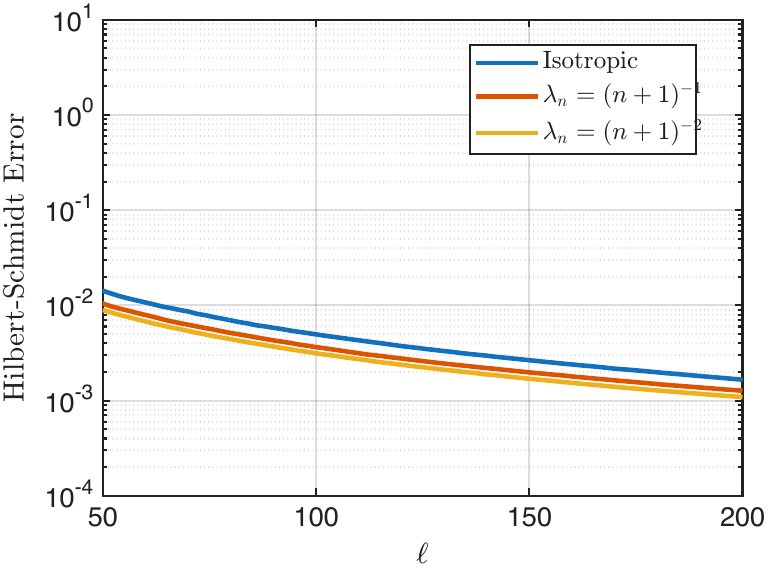}
    \caption{Performance of the randomised SVD when a Gaussian sketch with a decaying trigonometric covariance operator is used compared to isotropic. The case $\lambda_n = (n+1)^{-1}$ results in a compact but not trace-class covariance operator.}\label{fig:gaussian_cov_comparison}
\end{figure}

\begin{figure}[htbp]
    \centering
    \includegraphics[width=0.4\linewidth]{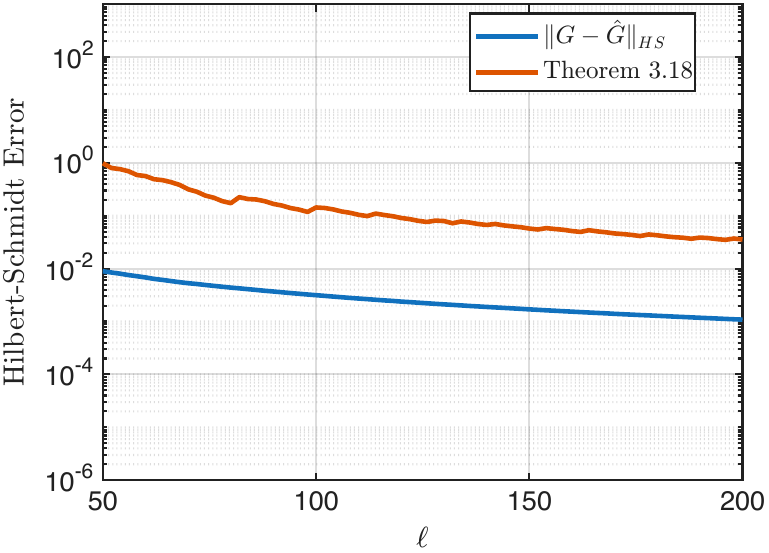}
    \includegraphics[width=0.4\linewidth]{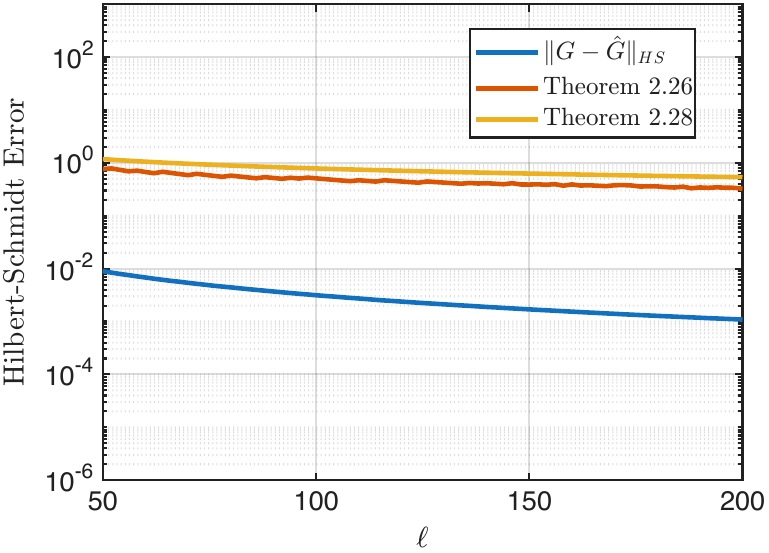}
    \caption{Incorporating prior knowledge of $G$ into the covariance operator can improve accuracy. Decaying trigonometric sketch with Rademacher (left) and Gaussian (right) coefficients and relevant theoretical bounds.}
    \label{fig:covariance_experiment}
\end{figure}
\subsection{The Discrete Schr\"odinger Operator}\label{subsection:discrete_schrod}

We can also consider $E$ on $\ell^2(\mathbb Z)$ if we replace $\Delta$ by the discrete Laplacian. In one dimension, this is given by $\Delta u_i = (u_{i+1}-u_i)-(u_i-u_{i-1})=u_{i+1}+u_{i-1}-2u_i$, the one-dimensional ``second derivative'' of a sequence. In higher dimensions, it is the sum of that quantity over all dimensions, like in the continuous case.

In this subsection, we consider the operator $E:D(E)\subseteq\ell^2(\mathbb Z)\to\ell^2(\mathbb Z)$ given by $Eu = -\Delta u + Vu$, with $V(n)=n^2$ and domain
\[
    D(E)=\big\{u\in\ell^2(\mathbb Z):(n^2u_n)_{n\in\mathbb Z}\in\ell^2(\mathbb Z)\big\}.
\]
The discrete Laplacian is bounded on $\ell^2(\mathbb Z)$, so the domain of $E$ is determined by the multiplication operator $V$. This operator has an infinite banded matrix representation.
\begin{equation}
    E = \begin{psmallmatrix} 
    \ddots & ~ & ~ & ~ & ~ & ~ & ~\\
    ~ &2+V(-2) & -1 & ~ & ~ & ~ & ~\\
    ~ & -1 & 2+V(-1) & -1 & ~ & ~ & ~\\
    \cdots & ~ & -1 & 2+V(0) & -1 & ~ & \cdots\\
    ~ & ~ & ~ & -1 & 2+V(1) & -1 & ~\\
    ~ & ~ & ~ & ~ & -1 & 2+V(2) & ~ \\
    ~ & ~ & ~ & ~& ~ & ~ & \ddots
    \end{psmallmatrix}
\end{equation}

In the previous subsection, we improved the randomised SVD by sampling from a distribution that decayed in the Fourier basis. The choice of basis was due to the observation that higher energy eigenfunctions corresponded to higher frequencies, so large components later in that basis. This is essentially due to $\sin(n\pi (x+1)/2)$ being an eigenvector of $-\Delta$ with eigenvalues scaling like $n^2$, subject to our boundary conditions.

Things are different in the case of the discrete Laplacian. Recall the matrix representation for $-\Delta$. It is a doubly infinite Toeplitz matrix (also called a Laurent operator), which has well-established spectral theory~\cite{bottcher2000toeplitz}.
\begin{equation}
    -\Delta = \begin{psmallmatrix} 
    \ddots & ~ & ~ & ~ & ~ & ~ & ~\\
    ~ &2 & -1 & ~ & ~ & ~ & ~\\
    ~ & -1 & 2& -1 & ~ & ~ & ~\\
    \cdots & ~ & -1 & 2 & -1 & ~ & \cdots\\
    ~ & ~ & ~ & -1 & 2 & -1 & ~\\
    ~ & ~ & ~ & ~ & -1 & 2 & ~ \\
    ~ & ~ & ~ & ~& ~ & ~ & \ddots
    \end{psmallmatrix}
\end{equation}
For the following formal calculation, we apply the discrete Laplacian coordinatewise to complex sequences that need not lie in $\ell^2(\mathbb Z)$. Define $\Gamma(t) = 2 - (t + t^{-1})$, and fix $z \in \mathbb C\setminus\{0\}$. Performing matrix multiplication with the sequence $v(z)=(z^i)_{i \in \mathbb Z}$, we see that

\begin{align*}
-\Delta v(z)&= \begin{psmallmatrix} 
    \ddots & ~ & ~ & ~ & ~ & ~ & ~\\
    ~ &2 & -1 & ~ & ~ & ~ & ~\\
    ~ & -1 & 2& -1 & ~ & ~ & ~\\
    \cdots & ~ & -1 & 2 & -1 & ~ & \cdots\\
    ~ & ~ & ~ & -1 & 2 & -1 & ~\\
    ~ & ~ & ~ & ~ & -1 & 2 & ~ \\
    ~ & ~ & ~ & ~& ~ & ~ & \ddots
    \end{psmallmatrix} \begin{psmallmatrix} \vdots \\ z^{-2} \\ z^{-1} \\z^0 \\ z^1 \\ z^2 \\ \vdots
    \end{psmallmatrix}\\
    &= \begin{psmallmatrix} \vdots \\ 2z^{-2} - (z^{-3} + z^{-1}) \\ 2z^{-1} - (z^{-2} + z^{0}) \\ 2z^{0} - (z^{-1} + z^{1}) \\ 2z^{1} - (z^{0} + z^{2}) \\ 2z^{2} - (z^{1} + z^{3}) \\ \vdots\end{psmallmatrix} 
    = \Gamma(z) \begin{psmallmatrix} \vdots \\ z^{-2} \\ z^{-1} \\z^0 \\ z^1 \\ z^2 \\ \vdots
    \end{psmallmatrix} = \Gamma(z) v(z)
\end{align*}
For $z^i$ to be bounded in both $i\gg 0$ and $i \ll 0$ it must be that $|z|=1$ i.e. $z \in \mathbb S^1$, and any such choice of $z$ gives an eigenvector in $\ell^{\infty}(\mathbb Z)$ but not in $\ell^2(\mathbb Z)$. So regarding bounded eigenvectors, the spectrum of $-\Delta$ is $\Gamma(\mathbb S^1)= [0, 4]$. Therefore in the discrete case, when $V(n)\to\infty$ this term will dominate the eigenvalues of $E$, which are $O(n^2)$ and our discrete $G$ is trace-class.

While the eigenvectors are in $\ell^\infty$ not $\ell^2$, when $-\Delta$ is restricted to $\ell^2$ the spectrum is still $\Gamma(\mathbb S^1)$. This is because the $\ell^\infty$ eigenvectors can be approximated by sequences of $\ell^2$ vectors, similar to the approach of Kressner et al.\cite{KRESSNER2026259}. Sketch functionals, however, are perfectly equipped to treat these directly. 

There is a simple centred and isotropic sketch which is supported on these eigenvectors; let $\Theta \sim U([0, 2 \pi]), \xi \sim U(\{-1,1\})$, and define $\ess(x) = \xi \sum_{n \in \mathbb  Z}x_n e^{in\Theta}$. A more intuitive interpretation is of $\mathbf{s} = \left(\xi e^{in\Theta}\right)_{n\in\mathbb{Z}}$ i.e.~sketch vectors that lie in $\ell^\infty$, but not within the Hilbert space $\ell^2$. This is an interesting case: the whole of $\ess$ is determined by the one-dimensional random variable $\Theta$ (plus a sign flip to make the sketch centred), yet it results in an infinite-dimensional isotropic random sketch supported on eigenvectors of $-\Delta$. Note that we have not explicitly discussed complex sketches like this in this paper (see Remark \ref{rem:complex}).

Despite its isotropy, this is a poor choice of distribution for randomised low-rank approximation. The amount of progress made by the approximation in the $\ell+1$th step depends on how linearly dependent $\ess_{\ell+1}$ is from the preceding sketch vectors $\ess_i, i \le \ell$. If two sketch vectors are highly linearly dependent, intuitively not much information is gained by sampling both. This is quantified by the $\Ess_1^\dagger$ term in the structural bound, as highly linearly dependent sketch vectors would result in low singular values of $\Ess_1$. In this case over many trials one is certain to sample $\theta_1, \theta_2$ close together, resulting in high linear dependence. 

Theoretically, $\ess$ is also about as poorly behaved as an isotropic sketch can be. Let $x \in \ell^2(\mathbb Z)$ and define $f_x \in L^2( [0, 2\pi) )$ by $f_x(\theta)= \sum_{n \in \mathbb Z} x_n e^{i n \theta}$, a well-defined function in $L^2$ by Fourier series with $\|f_x\|^2_{L^2} = 2\pi \|x\|^2_{\ell^2}$. Then $\ess(x) = \xi f_x(\Theta)$ pointwise. Note that by varying our choice of $x$, we can achieve any $f_x \in L^2$. We can thus guarantee that $\ess(x)$ has finite second moment:
\begin{equation}
    \mathbb E[|\ess(x)|^2]= \mathbb E[|f_x(\Theta)|^2] =  \int_0^{2\pi}|f_x (\theta)|^2 \frac{\mathrm d\theta}{2\pi} = \frac{1}{2\pi} \|f_x\|_{L^2}^2  = \|x\|^2_{\ell^2}\text{.}
\end{equation}
This also shows that $\ess$ is isotropic. However, we could easily have chosen some $f_x \in L^2([0, 2\pi) )$ but $f_x \not \in L^3([0, 2\pi) )$, and in this case we would have $\mathbb E[|\ess(x)|^3] = \mathbb E[|f_x(\Theta)|^3] = \infty$. So $\ess(x)$ is only guaranteed to have a finite second moment. It is interesting to compare this to the independent Rademacher case, whose components are also guaranteed to have magnitude $1$ but the stronger assumption of their independence both solves the linear dependence issue and gives strong, uniform guarantees on the distribution of $\ess(x)$ for all $x$. Certainly, none of our theoretical bounds apply, but this ``low-entropy sketch'' still serves as an interesting reference point for the worst-case scenario.

\begin{figure}[htbp]
    \centering
    \includegraphics[width=0.4\linewidth]{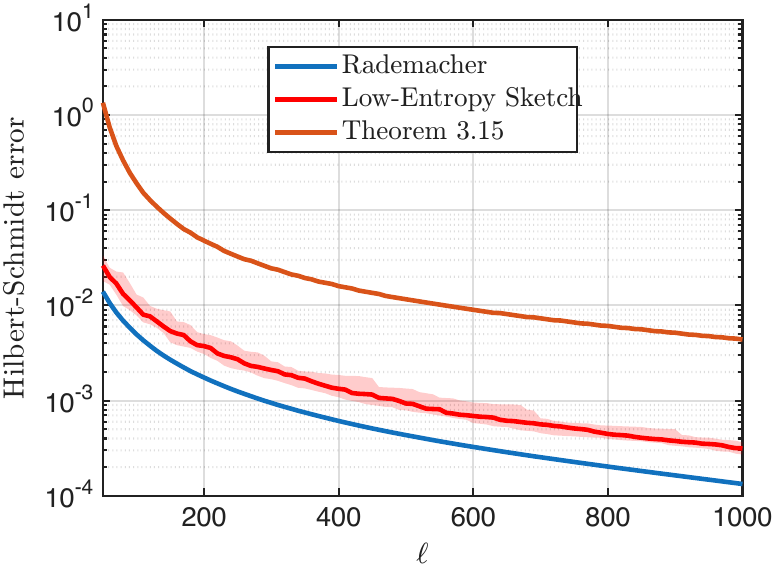}
    \includegraphics[width=0.4\linewidth]{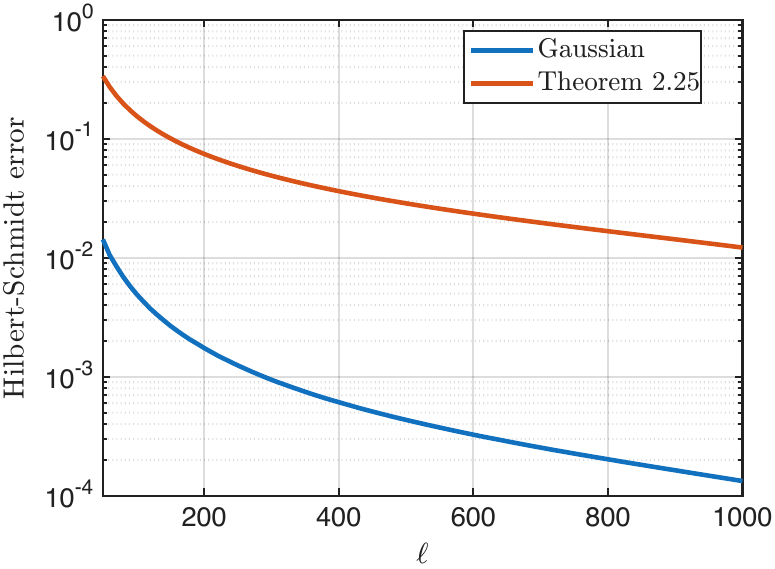}
    \caption{Hilbert--Schmidt error of the randomised SVD for the discrete $G$ using the isotropic Rademacher, isotropic Gaussian and low-entropy sketch. Rademacher and the low-entropy sketch are comparable in that they are isotropic and their components have magnitude 1, the latter performs significantly worse. While Theorem~\ref{theorem:subgaussian_maintheorem} is not applicable to the low-entropy sketch, it still appears to match the shape of the true median error.}
    \label{fig:rad_vs_eig}
\end{figure}

In Figure~\ref{fig:rad_vs_eig}, we plot the Hilbert--Schmidt error when the randomised SVD is used for $G$ with the low-entropy, isotropic Rademacher, and isotropic Gaussian sketches. The Gaussian and Rademacher distributions consistently perform indistinguishably well. Contrastingly, the variance of the error when the low-entropy sketch is used is much higher due to its dependence on a relatively small number of independent random variables. We have plotted the median with error bands here, but over one trial the error of the low-entropy sketch would plateau with increasing frequency, as significant progress is only made when the sampled $\Theta_k$ is not close to any previous $\Theta_i$. For this reason, notice the large error bands, which are almost invisible on other plots in this paper.

\section{Conclusions}

Our primary contribution is the application of cylindrical random variables~\cite{vakhania1987probability} to infinite-dimensional randomised NLA through the language of sketch functionals. With this tool, we have been able to unite and generalise previous results by Boull\'e--Townsend~\cite{boulle2021generalization} and Kressner et al.~\cite{KRESSNER2026259}, to analyse the randomised SVD for sketches with bounded covariance operators rather than being restricted to trace-class or the identity. Further, we have applied the results of Webb--Lynch-Davies \cite{webb2026sharp} to infinite dimensions yielding the first results on the infinite-dimensional randomised SVD with sub-Gaussian sketch. We argue all of this is evidence that sketch functionals are a natural framework for reasoning about infinite-dimensional random sketches. We have only given results on separable Hilbert spaces in this paper, but cylindrical random variables generalise to other Banach spaces and could have wider applications.

Such a notion certainly has the potential to broaden the horizon of infinite-dimensional randomised NLA; the first step of many randomised algorithms is to take the product of the target operator with a (usually isotropic) sketch matrix. Possible future applications of sketch functionals include lifting entirely different algorithms to infinite dimensions, for example randomised least squares~\cite{drineas2011faster}\cite{avron2010blendenpik}\cite{rokhlin2008fast} and randomised trace estimation~\cite{epperly2024xtrace}.

Our results share the limitation of Saibaba--Mi{\k{e}}dlar~\cite{saibaba2025randomized} that in the large $\ell$ limit, our bounds give $\|A - \hat A\|^2 \le 2\| \Sigma_\perp\|^2$. Comparable Gaussian results have the factor of $2$ tending instead to $1$; this difference is possibly due to the use of a conditional probability argument for the Gaussian case, but a crude submultiplicativity bound in the sub-Gaussian case because $\Ess_1$ and $\Ess_2$ are not necessarily independent. Perhaps more realistic tail bounds are possible with a different method of proof. For example, following Boutsidis~\cite{boutsidis2013} and writing $\Sigma_\perp \Ess_2 \Ess_1^\dagger = (\Sigma_\perp \Ess_2 \Ess_1^T)(\Ess_1 \Ess_1^T)^{-1}$ to treat cross-cancellation between $\Ess_2$ and $\Ess_1$ seems promising.

The theory in this paper has been focussed on the case where each sketch functional is independent from the others. In practice, this is not always the case: a particularly clever finite-dimensional sketch is the subsampled randomised Fourier transform (SRFT). It is a choice of sketch matrix $\Omega\in \mathbb R^{n \times \ell}$ which allows $Y=A\Omega$ to be computed highly efficiently, in $O(mn \log \ell)$ time for $A \in \mathbb R^{m \times n}$ as opposed to $O(m n \ell)$. It is not immediately clear what the infinite-dimensional analogue would be, but its columns could likely be captured by sketch functionals.

\section*{Tool and Computational Resource Disclosure}

Following the Leiden Declaration on Artificial Intelligence and Mathematics, we disclose that Claude (Sonnet 5, Anthropic) and Codex (GPT-6, OpenAI) was used for copyediting, and referee-style reports were produced by Claude and Codex. A number of suggestions were adopted including correcting typos and notation errors, fixing mathematical mistakes and pointing us to relevant literature. Codex was used to produce an initial draft of the proof of Lemma \ref{lem:sub-exp bound}. Opus 5.5 was also used to modify MATLAB code for figures in Section~\ref{section:numerical_exp}. All suggestions and contributions were verified by the authors, who take full responsibility for the content.

\bibliographystyle{plain}
\bibliography{subgaussian} 

\end{document}